\documentclass{article}
\usepackage{geometry}
\usepackage{amsmath,amsfonts, amsthm,amssymb,oldgerm}
\usepackage{graphicx}
\usepackage{array}
\usepackage{epsfig}
\usepackage[pdftex,bookmarksopen=true,colorlinks=true, linkcolor=red, citecolor=blue]{hyperref}
\usepackage{relsize}
\def\r{\mathbb{R}}

\def\<{\langle}
\def\>{\rangle}

\def\2{L^2}
\def\c{L^2([0,T];H^1(0,1))}

\def\h{\hat}

\def\lam{\lambda}
\def\ds{\displaystyle}
\newtheorem{thm}{\bf Theorem}[section]

\newtheorem{lem}[thm]{\bf Lemma}
\newtheorem{cor}[thm]{\bf Corollary}
\newtheorem{rem}[thm]{\bf Remark}

\newtheorem{prop}[thm]{\bf{Proposition}}

\begin{document}
\title{General Boundary Value Problems of
 the Korteweg-de Vries Equation on a Bounded Domain}

\author{R. A. Capistrano-Filho\footnote{Departamento de Matem\'atica, Universidade Federal de Pernambuco, Recife - Pernambuco, 50740-545, Brazil - capistranofilho@dmat.ufpe.br}, \
Shu-Ming Sun\footnote{Department of Mathematics, Virginia Tech,
Blacksburg, VA 24061 - sun@math.vt.edu} \ and
Bing-Yu Zhang\footnote{Department of Mathematical Sciences, University of Cincinnati,  Cincinnati, Ohio 45221-0025,
  United States - zhangb@ucmail.uc.edu}}

\date{}

\maketitle

\begin{abstract}
In this paper we consider the initial boundary value problem of  the Korteweg-de Vries equation  posed on a finite interval
\begin{equation}\label{a.b}
 u_t+u_x+u_{xxx}+uu_x=0,\qquad u(x,0)=\phi(x), \qquad  0<x<L,  \
 t>0
\end{equation}
subject to  the nonhomogeneous boundary conditions,
\begin{equation} \label{a.b_1}
 B_1u=h_1(t), \qquad B_2 u= h_2 (t),  \qquad B_3 u= h_3 (t) \qquad t>0
 \end{equation} where
\[ B_i u =\sum _{j=0}^2 \left(a_{ij} \partial ^j_x u(0,t) + b_{ij}
\partial ^j_x u(L,t)\right), \qquad i=1,2,3,\]
and $a_{ij}, \ b_{ij}$ $ (j,i=0, 1,2,3)$ are real constants.  Under some   general   assumptions imposed on the coefficients  $a_{ij}, \ b_{ij}$, $ j,i=0, 1,2,3$,    the IBVPs (\ref{a.b})-(\ref{a.b_1})
 is shown to be  locally well-posed in the space $H^s (0,L)$ for any $s\geq 0$  with $\phi \in H^s (0,L)$ and   boundary values $h_j, j=1,2,3$ belonging to some appropriate spaces with  optimal regularity.
\end{abstract}

\section{Introduction}\label{sec1}
\setcounter{equation}{0}

In this paper we consider the initial-boundary value problems (IBVP) of the
Korteweg-de Vries (KdV) equation posed on a finite domain $(0,L)$
\begin{equation}\label{1.1}
 u_t+u_x+u_{xxx}+uu_x=0,\qquad u(x,0)=\phi(x), \qquad  0<x<L,  \
 t>0
    \end{equation}
with  general   non-homogeneous boundary conditions posed on the two ends of the
domain  $(0,L)$,
\begin{equation} \label{1.2}
 B_1u=h_1(t), \qquad B_2 u= h_2 (t),  \qquad B_3 u= h_3 (t) \qquad
 t>0,
 \end{equation} where
\[ B_i u =\sum _{j=0}^2 \left(a_{ij} \partial ^j_x u(0,t) + b_{ij}
\partial ^j_x u(L,t)\right), \qquad i=1,2,3,\]
and $a_{ij}, \ b_{ij}$, $ j, \,i=0, 1,2,$ are real constants.

We are mainly concerned with the following question:

 \medskip
\emph{Under what assumptions on the coefficients $a_{kj}, \ b_{kj}
$ in (\ref{1.2}) is the IBVP (\ref{1.1})-(\ref{1.2}) well-posed
in the classical Sobolev space $H^s (0,L)$?}
 \medskip

%A review of the mathematical theory currently available is now presented. For the KdV equation on a finite interval,
 As early as in 1979, Bubnov \cite{Bubnov79} studied  the following IBVP of
the KdV equation on the finite interval $(0,1)$:
\begin{equation}\label{1.3}
\begin{cases}
 u_t +uu_x+u_{xxx}=f, \quad u(x,0)=0, \quad x\in (0,1), \ t\in
 (0,T), \\ \alpha _1 u_{xx}(0,t)+\alpha _2 u_x (0,t)+\alpha _3
 u(0,t)=0, \\ \beta_1 u_{xx} (1,t)+\beta _2 u_x (1,t)+ \beta _3
 u(1,t) =0, \\ \chi _1 u_x (1,t)+ \chi _2 u(1,t)=0
 \end{cases}
 \end{equation}
and obtained the result as described below.

 \medskip
 \noindent
 {\bf Theorem A }\cite{Bubnov79}:
\emph{Assume that
\begin{equation}\label{1.4}
\begin{cases}
if \ \alpha _1 \beta_1 \chi _1 \ne 0, \ then \ F_1>0, \ F_2 >0, \\
if  \ \beta _1\ne 0, \ \chi _1 \ne 0, \ \alpha _1 =0, \ then \
\alpha _2=0, \ F_2 >0, \ \alpha _3 \ne 0, \\ if \ \beta _1 =0, \
\chi _1 \ne 0, \ \alpha _1 \ne 0, \ then \ F_1 >0, \ F_3 \ne 0, \\
if \ \alpha _1=\beta _1 =0, \ \chi _1 \ne 0, \ then \ F_3\ne 0, \
\alpha _2 =0, \ \alpha _3 \ne 0, \\ if \ \beta _1 =0, \ \alpha _1
\ne 0, \ \chi _1 =0, \ then \ F_1 >0, \ F_3 \ne 0, \\ if \ \alpha
_1=\beta _1 =\chi _1 =0, \ then \ \alpha _2 =0, \ \alpha _3 \ne 0, \
F_3 \ne 0,
\end{cases}
\end{equation}
where
\[ F_1=\frac{\alpha _3}{ \alpha _1} -\frac{\alpha _2^2}{2\alpha
_1^2}, \ F_2 =\frac{\beta_2 \chi _2}{\beta _1 \chi _1}
-\frac{\beta _3}{\beta _1} -\frac{\chi _2^2}{2\chi _1^2}, \ F_3
=\beta _2 \chi_2-\beta _1\chi _1 . \]  For any given
\[ f\in H^1_{loc}(0, \infty ; L^2 (0,1)) \ with \ \ f(x,0)=0, \]
 there exists a $T>0$ such that (\ref{1.3}) admits a unique
solution
\[ u\in L^2 (0, T; H^3 (0,1)) \ with  \ u_t \in L^{\infty} (0,T;
L^2 (0,1))\cap L^2 (0,T; H^1 (0,1)) .\] }
%\begin{rem} A little calculation can show that the assumptions (\ref{1.4}) on the boundary conditions
%considered by Bubnov lead to the following four types  of boundary
%conditions.
%\begin{itemize}
%\item[(a)] $u(0,t)=0, \ u(1,t)=0, \ u_x(1,t)=0$;
%\item[(b)] $u_{xx}(0,t)+au_x(0,t)+bu(0,t)=0, \ u(1,t)=0, \ u(1,t)=0$
%with $a>\frac{b^2}{2}$;
%\item[(c)] $u(0,t)=0, \ u_{xx}(1,t)+au_x(1,t)+bu(1,t)=0, \ u_x
%(1,t)+cu(1,t)=0$ with $ac > b-\frac{c^2}{2}$;
%\item[(d)] $u_{xx}(0,t)+a_1u_x(0,t)+b_1u(0,t)=0, \  u_{xx}(1,t)+a_2u_x(1,t)+b_2u(1,t)=0, \ u_x
%(1,t)+cu(1,t)=0 $ with $b_1> \frac{a_1^2}{2}, \ a_2c>b_2
%-\frac{c^2}{2}.$
%\end{itemize}
%\end{rem}
%\begin{rem}
 The main tool used by Bubnov \cite{Bubnov79} to prove this theorem is the
following Kato type smoothing property for  the solution $u$ of  the
linear system associated to the  IBVP (\ref{1.2}),
\begin{equation}\label{1.5}
\begin{cases}
 u_t  +u_{xxx}=f, \quad u(x,0)=0, \quad x\in (0,1), \ t\in
 (0,T), \\ \alpha _1 u_{xx}(0,t)+\alpha _2 u_x (0,t)+\alpha _3
 u(0,t)=0, \\ \beta_1 u_{xx} (1,t)+\beta _2 u_x (1,t)+ \beta _3
 u(1,t) =0, \\ \chi _1 u_x (1,t)+ \chi _2 u(1,t)=0.
 \end{cases}
 \end{equation}
 Under the
 assumptions (\ref{1.4}),
\begin{equation}\label{1.6}f\in L^2(0,T; L^2 (0,1))\implies  u\in
L^2 (0,T; H^1 (0,1))\cap L^{\infty} (0,T; L^2 (0,1))\end{equation}
and
\[ \|u\|_{L^2 (0,T; H^1 (0,1))}+ \|u\|_{L^{\infty} (0,T; L^2
(0,1))} \leq C\|f\|_{L^2 (0,T; L^2 (0,T))}, \] where $C>0$ is a
constant independent of $f$.
%\end{rem}

\medskip
  In the past thirty years since the work of
Bubnov \cite{Bubnov79}, various boundary-value problems of the KdV equation have
been studied. In particular, the following two  special classes of IBVPs of the
KdV equation on the finite interval $(0,L)$,

\begin{equation}\label{1.7}
\begin{cases}
u_t +u_x +uu_x +u_{xxx}=0, \ u(x,0)=\phi (x), \quad x\in (0,L), \
t>0, \\ u(0,t)= h_1(t), \quad u(L,t) = h_2 (t), \quad u_x (L,t)
=h_3 (t) \end{cases}
\end{equation}
and
\begin{equation}\label{1.8}
\begin{cases}
u_t +u_x +uu_x +u_{xxx}=0, \ u(x,0)=\phi (x), \quad x\in (0,L), \
t>0, \\ u(0,t)= h_1(t), \quad u(L,t) = h_2 (t), \quad u_{xx} (L,t)
=h_3 (t), \end{cases}
\end{equation}
as well as the IBVPs of the KdV equation posed in a quarter plane have been intensively studied in the past twenty years (cf. \cite{BSZ03FiniteDomain,bsz-finite,ColGhi97,ColGhi97a,ColGhi01,Fam83,Fam89,Holmer06, KrZh, KrIvZh, RiUsZh}  and the references therein)
following the rapid advances of the study of the pure initial value
problem of the KdV equation posed on the whole line $\r$ or on the
periodic domain $\mathbb{T}$ (see e.g. \cite{BS76, BS78,Bourgain93a,Bourgain93b, Bourgain97,ColKeel03,Fam83,Fam89,Fam99,KPV89,KPV91,KPV91-1,KPV93,KPV93b,KPV96}).

\medskip
The nonhomogeneous IBVP (\ref{1.7}) was first shown by Faminskii in \cite{Fam83,Fam89}
to be well-posed in the  spaces $L^2 (0,L)$ and $H^3 (0,L)$.

\medskip
\noindent {\bf Theorem B} \cite{Fam83,Fam89}: \emph{Let $T>0$ be given. For
any $\phi \in L^2 (0,L)$ and
\[\vec{h}= (h_1, h_2, h_3) \in W^{\frac13, 1}(0,T)\cap
L^{6+\epsilon} (0,T)\cap H^{\frac16} (0,T)\times W^{\frac56
+\epsilon, 1} (0,T)\cap H^{\frac13} (0,T)\times  L^2 (0,T),\]
the IBVP (\ref{1.7}) admits a unique solution $u\in C([0,T]; L^2
(0,L))\cap L^2 (0,T; H^1 (0,L))$. Moreover, the solution map is
continuous in the corresponding spaces.}
  \emph{In addition, if $\phi \in
H^3 (0,L)$, $$ h_1' \in W^{\frac13, 1}(0,T)\cap L^{6+\epsilon}
(0,T)\cap H^{\frac16} (0,T),$$
 $$ h_2'\in W^{\frac56 +\epsilon, 1} (0,T)\cap H^{\frac13 } (0,T)$$  and $$ h_3' \in L^2 (0,T)$$
 with
$$ \phi (0)=h_1 (0), \phi (L)=h_2 (0), \ \phi' (L) = h_3 (0),$$
then the solution $u\in C^1([0,T]; H^3 (0,L))\cap L^2 (0,T; H^4(0,L))$.}

\medskip
Bona \textit{et al.}, in \cite{BSZ03FiniteDomain}, showed that IBVP (\ref{1.7}) is
locally well-posed in the space $H^s (0,L)$ for any $s\geq 0$.

\medskip
\noindent {\bf Theorem C} \cite{BSZ03FiniteDomain}:
 \emph{Let $s\geq
0$ , $r>0$ and $T>0$ be given.}  \emph{There
exists $T^*\in (0, T]$ such that for  any $s-$compatible
\footnote{See \cite{BSZ03FiniteDomain} for exact definition of $s-$compatibility.}
\[ \phi \in H^s (0,L),  \quad \vec{h}= (h_1, h_2, h_3) \in
H^{\frac{s+1}{3}}(0,T) \times H^{\frac{s+1}{3}}(0,T)\times
H^{\frac{s}{3}} (0,T) \] satisfying
\[ \| \phi \| _{H^s (0,L)} + \| \vec{h}\|_{H^{\frac{s+1}{3}}(0,T) \times H^{\frac{s+1}{3}}(0,T)\times
H^{\frac{s}{3}} (0,T)} \leq r,\] the IBVP (\ref{1.7}) admits a
unique solution
\[ u\in C([0,T^*]; H^s (0,L))\cap L^2 (0,T^*; H^{s+1} (0,L)).\]
Moreover, the corresponding solution map is  analytic
in the corresponding spaces.}

\medskip
Later on, in \cite{Holmer06}, Holmer showed that the IBVP (\ref{1.7}) is locally
well-posed in the space $H^s (0,L)$, for any $-\frac34 <s<
\frac12$,  and Bona \textit{et al.}, in \cite{bsz-finite}, showed that the IBVP
(\ref{1.7}) is locally well-posed $H^s (0,L)$ for any $s>-1$.

 %\medskip
%\noindent {\bf Theorem D} \cite{bsz-finite}
 %\emph{Let $r>0$, $-1< s\leq 0$ and $T>0$ be given.}
%\emph{There exists a $T^*\in (0, T]$ such that for any
%\[ \phi \in H^s (0,L),  \quad \vec{h}= (h_1, h_2, h_3) \in
%H^{\frac{s+1}{3}}(0,T) \times H^{\frac{s+1}{3}}(0,T)\times
%H^{\frac{s}{3}} (0,T) \] satisfying
%\[ \| \phi \| _{H^s (0,L)} + \| \vec{h}\|_{H^{\frac{s+1}{3}}(0,T) \times H^{\frac{s+1}{3}}(0,T)\times
%H^{\frac{s}{3}} (0,T)} \leq r,\] the IBVP (\ref{1.7}) admits a
%unique mild solution
%\[ u\in C([0,T^*]; H^s (0,L)).\]
%Moreover, the corresponding solution map is analytic
%in the corresponding spaces.}

\medskip
%Several points are also worth noting about the well-posedness of the system \eqref{1.8}.
As for the IBVP (\ref{1.8}), its study began with the work of Colin and Ghidalia in late 1990's \cite{ColGhi97,ColGhi97a,ColGhi01}.
They obtained  in  \cite{ColGhi01}  the following results.

\medskip
 \noindent
 {\bf Theorem D} \cite{ColGhi01}:
 \begin{itemize} \item[(i)]  \emph{Given
$h_j\in C^1([0, \infty)), \ j=1,2,3$ and $\phi \in H^1 (0,L)$
satisfying $h_1(0)=\phi (0)$, there exists a $T>0$ such that the
IBVP (\ref{1.8}) admits a solution (in the sense of distribution)}
\[ u\in L^{\infty}(0,T; H^1(0,L))\cap C([0,T]; L^2 (0,L)) .\]

\item[(ii)] \emph{The solution $u$ of the IBVP (\ref{1.8}) exists
globally in $H^1(0,L)$ if the size of its initial value $\phi \in
H^1 (0,L)$ and its boundary values $h_j\in C^1([0, \infty )), \
j=1,2,3$ are all small.}
\end{itemize}
In addition, they showed that the associate linear IBVP
\begin{equation}\label{1.9}
        \begin{cases}
        u_t+u_x+u_{xxx}=0,\qquad u(x,0)=\phi(x)  & x\in (0,L),  \ t\in \r^+ \\
        u(0,t)=0,\ u_x(L,x)=0,\ u_{xx}(L,t)=0
        \end{cases}
    \end{equation}
 possesses a strong smoothing property:
 \par{
\emph{ For any $\phi \in L^2 (0,L)$, the linear IBVP (\ref{1.9})
admits a
 unique solution $$u\in C(\r^+; L^2 (0,L))\cap L^2 _{loc} (\r^+; H^1
 (0,L)).$$}}
\noindent Aided by this smoothing property, Colin \textit{et al.} showed
that the homogeneous IBVP (\ref{1.9}) is locally well-posed in the
space $L^2 (0,L)$.

 \medskip
 \noindent
 {\bf Theorem E} \cite{ColGhi01}:  \emph{For any given $\phi \in L^2
(0,L)$, there exists a $T>0$ such that the IBVP (\ref{1.9}) admits
a unique weak solution $u\in C([0,T]; L^2 (0,L))\cap L^2 (0,T; H^1
(0,L))$.}
\medskip

Recently, Kramer \textit{et al.}, in \cite{KrIvZh}, and Jia \textit{et al.}, in \cite{JIZ},  have shown that the IBVP \eqref{1.8} is locally well-posedness in the classical Sobolev space $H^s(0,L)$ for $s>-1$, which provide a positive answer to one of the open question in \cite{ColGhi01}.

\medskip
\noindent {\bf Theorem F} \cite{JIZ,KrIvZh}:  \emph{Let $s>-1$, $T>0$  and
$r>0$ be given with
$$s\neq\frac{2j-1}{2}, \text{} \text{}j=1,2,3...$$There exists a $0<T^*\leq T$    such that for any
  $\phi \in H^s (0,L)$, \[ h_1\in H^{\frac{s+1}{3}}
(0,T), \ h_2 \in H^{\frac{s}{3}} (0,T), \ h_3 \in H^{\frac{s-1}{3}}
(0,T)\] satisfying
\[ \| \phi \|_{H^s (0,L)}+\|h_1\|_{H^{\frac{s+1}{3}} (0,T)}+\|h_2\|_{H^{\frac{s }{3}}
(0,T)}+\|h_3\|_{H^\frac{s-1}{3} (0,T)} \leq r,\] the IBVP (\ref{1.8})
 admits a unique mild  solution
\[ u\in C([0,T]; H^s (0,L)). \]
Moreover, the corresponding solution map is analytically
continuous.}

In addition,  Rivas \textit{et al.}, in \cite{RiUsZh}, shown that the solutions of the IBVP (\ref{1.8})  exist globally
as long as their initial value and the associate boundary data are small. Moreover, those solutions decay exponentially if their boundary data decay exponentially.

 \medskip
 \noindent
 {\bf Theorem G} \cite{RiUsZh}: \emph{Let $s\geq 0$ with
$$s\neq\frac{2j-1}{2}, \text{} \text{}j=1,2,3...$$
There exist positive constants $\delta$ and $T$ such that for any $s-$compatible
 \footnote{See \cite{RiUsZh} for exact definition, in this case, of $s-$compatibility.} $\phi \in H^s (0,L)$ and $$\vec{h}= (h_1, h_2, h_3) \in
B^s_{(t,t+T)}:=H^{\frac{s+1}{3}}(t,t+T) \times H^{\frac{s}{3}}(t,t+T)\times
H^{\frac{s-1}{3}} (t,t+T) $$ for any $t\geq 0$
 with
 \[ \|\phi \|_{H^s (0,L)} + \sup_{t\geq 0}  \|\vec{h}\|_{B^s_{(t,t+T)}} \leq \delta , \]
   the IBVP (\ref{1.8}) admits a unique solution
 \[ u\in Y^s_{(t,t+T)}:=C([t,t+T]; H^s (0,L))\cap L^2 (t,t+T; H^{s+1}(0,L))\] for any $t\geq0,$ and
  \[\sup_{t\geq0}\|\vec{v}\|_{Y^s_{(t,t+T)}}<\infty.\]
 If, in addition to these conditions, there exist $\gamma_1>0$, $C_1>0$ and $g\in B^s_{(t,t+T)}$ such that
 \[\sup_{t\geq0}\|g\|_{B^s_{(t,t+T)}}<\infty,\]
 and
 $$\|\vec{h}\|_{B^s_{(t,t+T)}}\leq g(t)e^{-\gamma_1t}, \text{}\text{ for } t\geq0,$$
then there exists $\gamma$ with $0<\gamma\leq\gamma_1$ and $C_2>0$ such that the corresponding solution $u$ of the IBVP \eqref{1.8} satisfies
$$\|u\|_{Y^s_{(t,t+T)}}\leq C_2 (\|\phi \|_{H^s (0,L)} + \|\vec{h}\|_{B^s_{(t,t+T)}})e^{-\gamma t}, \text{}\text{ for } t\geq0.$$}
 \medskip

 As for the general IBVP (\ref{1.1})-(\ref{1.2}), Kramer and Zhang, in \cite{KrZh}, studied the following non-homogeneous
boundary value problem,
\begin{equation}\label{1.3-g}
\begin{cases}
 u_t +uu_x+u_{xxx}=0, \quad u(x,0)=\phi (x), \quad x\in (0,1), \ t\in
 (0,T), \\ \alpha _1 u_{xx}(0,t)+\alpha _2 u_x (0,t)+\alpha _3
 u(0,t)=h_1(t), \\ \beta_1 u_{xx} (1,t)+\beta _2 u_x (1,t)+ \beta _3
 u(1,t) =h_2(t), \\ \chi _1 u_x (1,t)+ \chi _2 u(1,t)=h_3 (t).
 \end{cases}
 \end{equation}
 They showed that the IBVP (\ref{1.3-g}) is locally well-posed in
 the space $H^s (0,1)$, for any $s\geq 0$, under the assumption (\ref{1.4}).

 \medskip
 \noindent
 {\bf Theorem H} \cite{KrZh}: \emph{Let $s\geq 0$ with
$$s\neq\frac{2j-1}{2}, \text{} \text{}j=1,2,3...,$$ $T>0$ be given and assume (\ref{1.4}) holds. For any $r>0$,
 there exists a $T^*\in (0,T]$ such that for any $s-$compatible
 \footnote{See \cite{KrZh} for exact definition, in this case, of $s-$compatibility.} $\phi \in H^s (0,1)$, $h_j\in H^{\frac{s+1}{3}}(0,T),
 j=1,2,3$ with
 \[ \|\phi \|_{H^s (0,1)} + \|h_1\|_{H^{\frac{s+1}{3}}(0,T)}
 +\|h_2\|_{H^{\frac{s+1}{3}}(0,T)}+\|h_3\|_{H^{\frac{s+1}{3}}(0,T)}
 \leq r,\]  the IBVP (\ref{1.3-g}) admits a unique solution
 \[ u\in C([0,T^*]; H^s (0,1))\cap L^2 (0,T^*; H^{s+1}(0,1)) .\]
 Moreover, the solution $u$ depends continuously on its initial data
 $\phi $ and the boundary values $h_j, j=1,2,3,$ in the respective
 spaces.}

 \medskip

In this paper we continue to study   the  general IBVP (\ref{1.1})-(\ref{1.2})  for its well-posedness in the space $H^s (0,L)$
and attempt to provide a (partial) answer  asked earlier,

 \medskip
\emph{Under what assumptions on the coefficients $a_{kj}, \ b_{kj}
$ in (\ref{1.2}) is the IBVP (\ref{1.1})-(\ref{1.2}) well-posed
in the classical Sobolev space $H^s (0,L)$?}
 \medskip

We propose the following hypotheses on those coefficients $a_{ij},
\ b_{ij}$, $ j,i=0, 1,2,3$:
\begin{itemize}
\item[(A1)] $ a_{12}=a_{11}=0, \ a_{10}\ne0, \
b_{12}=b_{11}=b_{10}=0$;

\item[(A2)] $a_{12}\ne0, \ b_{12}=0$;

\item[(B1)] $\ a_{22}=a_{21}=a_{20}
=0,\ b_{20}\ne0, b_{22}=b_{21}=0$;

\item[(B2)] $b_{22}\ne 0, \ a_{22}=0 $;

\item[(C)] $a_{32}=a_{31}=0, \ b_{31}\ne 0, \ b_{32}=0.$
\end{itemize}
In addition,  for any  $s\geq 0$,  $$H^s_0(0,T]:=\{h(t)\in H^s(0,T):h^{(j)}(0)=0\},$$
for $j=0,1,...,[s] $. Let us consider
\begin{equation}
\label{x-1}
\begin{cases}
 {\cal H}_1^s (0,T) := H_0^{\frac{s+1}{3}}(0,T]\times
H_0^{\frac{s+1}{3}}(0,T]\times H_0^{\frac{s}{3}}(0,T], \\ {\cal H}^s_2 (0,T):= H_0^{\frac{s+1}{3}}(0,T]\times
H_0^{\frac{s}{3}}(0,T]\times H_0^{\frac{s-1}{3}}(0,T],\\
{\cal H}^s_3 (0,T):= H_0^{\frac{s-1}{3}}(0,T]\times
H_0^{\frac{s+1}{3}}(0,T]\times H_0^{\frac{s}{3}}(0,T], \\ {\cal H}^s_4 (0,T):= H_0^{\frac{s-1}{3}}(0,T]\times
H_0^{\frac{s-1}{3}}(0,T]\times H_0^{\frac{s}{3}}(0,T]
\end{cases}
\end{equation}
and
\begin{equation}
\label{x-1-1}
\begin{cases}
 {\cal W}_1^s (0,T) := H^{\frac{s+1}{3}}(0,T)\times
H^{\frac{s+1}{3}}(0,T)\times H^{\frac{s}{3}}(0,T), \\ {\cal W}^s_2 (0,T):= H^{\frac{s+1}{3}}(0,T)\times
H^{\frac{s}{3}}(0,T)\times H^{\frac{s-1}{3}}(0,T),\\
{\cal W}^s_3 (0,T):= H^{\frac{s-1}{3}}(0,T)\times
H^{\frac{s+1}{3}}(0,T)\times H^{\frac{s}{3}}(0,T), \\ {\cal W }^s_4 (0,T):= H^{\frac{s-1}{3}}(0,T)\times
H^{\frac{s-1}{3}}(0,T)\times H^{\frac{s}{3}}(0,T).
\end{cases}
\end{equation}

% where   $h^{(j)}(t)$ is the $j$--th order derivative of $h$.
 \medskip

We have the following well-posedness results for the  IBVP
(\ref{1.1})-(\ref{1.2}).
\begin{thm}\label{th1}  Assume (A1), (B1) and (C) hold and let $s\geq 0$ with
$s\neq\frac{2j-1}{2}, \text{} \text{}j=1,2,3..., $ and
$T>0$ be given. Then  for any $r>0$ there exists a $T^*\in (0, T]$
 such that for any $$(\phi , \vec{h})\in H^s_0
(0,L)\times{\cal H}^s_1(0,T)$$ satisfying

  % \begin{equation}\label{scp1}
    %\begin{cases}
    %\phi(0)=\phi(L)=0,& \text{ if }\text{ } \frac{1}{2}<s\leq3,\\
    %\phi(L,t)+\phi'(L)=0,& \text{ if }\text{ } \frac{3}{2}<s\leq3,
    %\end{cases}
    %\end{equation}
%and
\[ \|(\phi, \vec{h})\|_{L^2 (0,L)\times{\cal H}^0_1(0,T)} \leq r\] the IBVP (\ref{1.1})-(\ref{1.2})
admits a solution
$$u\in C([0,T^*]; H^s (0,L))\cap L^2 (0,T^*;H^{s+1}(0,L))$$
possessing the hidden regularities (the sharp Kato smoothing properties)
$$\partial_x^lu\in L^{\infty}(0,L;H^{\frac{s+1-l}{3}}(0,T^*)) \text{ for }l=0,1,2.$$
Moreover, the corresponding solution map is analytically  continuous.
\end{thm}

\begin{thm}\label{th2}  Assume (A1), (C) and (B2) hold and let $s\geq 0$ with
$s\neq\frac{2j-1}{2}, \text{} \text{}j=1,2,3..., $ and
$T>0$ be given. Then  for any $r>0$ there exists a $T^*\in (0, T]$
 such that for any $$(\phi , \vec{h})\in H_0^s
(0,L)\times {\cal H}^s_2(0,T) $$satisfying

    %\begin{equation}\label{scp2}
    %\begin{cases}
    %\phi(0)=0,& \text{ if }\text{ } \frac{1}{2}<s\leq3,\\
    %\phi(L)+\phi'(L)=0,& \text{ if }\text{ } \frac{3}{2}<s\leq3, \\
    %\phi(L)+\phi'(L)+\phi''(L)=0, & \text{ if }\text{ } \frac{5}{2}<s\leq3,
        %\end{cases}
    %\end{equation}
%and
\[ \|(\phi, \vec{h})\|_{L^2 (0,L)\times
{\cal H}^0_2(0,T)} \leq r\] the IBVP (\ref{1.1})-(\ref{1.2})
admits a solution $$u\in C([0,T^*]; H^s (0,L))\cap L^2 (0,T^*;
H^{s+1}(0,L))$$possessing the hidden regularities (the sharp Kato smoothing properties)
$$\partial_x^lu\in L^{\infty}(0,L;H^{\frac{s+1-l}{3}}(0,T^*)) \text{ for }l=0,1,2.$$
Moreover, the corresponding solution map is analytically  continuous.
\end{thm}

\begin{thm}\label{th3}  Assume (A2), (B1) and (C) hold and let $s\geq 0$  with
$s\neq\frac{2j-1}{2}, \text{} \text{}j=1,2,3..., $ and
$T>0$ be given. Then  for any $r>0$ there exists a $T^*\in (0, T]$
 such that for any $$(\phi , \vec{h})\in H_0^s
(0,L)\times {\cal H}^s_3 (0,T)$$ satisfying

% \begin{equation}\label{scp3}
   %   \begin{cases}
      % \phi(0)=\phi(L)=0,& \text{ if }\text{ } \frac{1}{2}<s\leq3,\\
       %\phi'(L)=0& \text{ if }\text{ } \frac{3}{2}<s\leq3,\\
       %\phi(0)+\phi'(0)+\phi''(0)=0, & \text{ if }\text{ } \frac{5}{2}<s\leq3,
        %\end{cases}
    %\end{equation}
%and
\[ \|(\phi, \vec{h})\|_{L^2 (0,L)\times
{\cal H}^0_3 (0,T)} \leq r\] the IBVP (\ref{1.1})-(\ref{1.2})
admits a solution $$u\in C([0,T^*]; H^s (0,L))\cap L^2 (0,T^*;
H^{s+1}(0,L))$$possessing the hidden regularities (the sharp Kato smoothing properties)
$$\partial_x^lu\in L^{\infty}(0,L;H^{\frac{s+1-l}{3}}(0,T^*)) \text{ for }l=0,1,2.$$
Moreover, the corresponding solution map is analytically continuous.
\end{thm}

\begin{thm}\label{th4}  Assume (A2), (C) and (B2) hold and let $s\geq 0$  with
$s\neq\frac{2j-1}{2}, \text{} \text{}j=1,2,3..., $ and
$T>0$ be given. Then  for any $r>0$ there exists a $T^*\in (0, T]$
 such that for any $$(\phi , \vec{h})\in H_0^s
(0,L)\times {\cal H}^s_4 (0,T)$$ satisfying

% \begin{equation}\label{scp3}
   %   \begin{cases}
      % \phi(0)=\phi(L)=0,& \text{ if }\text{ } \frac{1}{2}<s\leq3,\\
       %\phi'(L)=0& \text{ if }\text{ } \frac{3}{2}<s\leq3,\\
       %\phi(0)+\phi'(0)+\phi''(0)=0, & \text{ if }\text{ } \frac{5}{2}<s\leq3,
        %\end{cases}
    %\end{equation}
%and
\[ \|(\phi, \vec{h})\|_{L^2 (0,L)\times
{\cal H}^0_4 (0,T)} \leq r\]  the IBVP (\ref{1.1})-(\ref{1.2})
admits a solution $$u\in C([0,T^*]; H^s (0,L))\cap L^2 (0,T^*;
H^{s+1}(0,L))$$possessing the hidden regularities (the sharp Kato smoothing properties)
$$\partial_x^lu\in L^{\infty}(0,L;H^{\frac{s+1-l}{3}}(0,T^*)) \text{ for }l=0,1,2.$$
Moreover, the corresponding solution map is analytically continuous.
\end{thm}

The following remarks are now in order.

\smallskip
{\it
\begin{itemize}
\item[(i)] The  temporal regularity conditions imposed on the boundary values  $\vec{h}$ are optimal (cf. \cite{BSZ02, BSZ04, BSZ06}).
\item[(ii)]  The assumptions imposed on the boundary conditions in  the Theorems \ref{th1}-\ref{th4} can be reformulated  as follows:
\begin{itemize}
\item[(i)]  $((A1), (B1), (C)) \Leftrightarrow \mathcal{B}_{1}v= \vec{h},$
\item[(ii)]  $((A_1),(C),(B_2)) \Leftrightarrow \mathcal{B}_{2}v= \vec{h},$
\item[(iii)]  $((A_2), (B_1), (C)) \Leftrightarrow \mathcal{B}_{3}v= \vec{h},$
\item[(iv)]  $((A_2),(C),(B_2)) \Leftrightarrow \mathcal{B}_{4}v= \vec{h}.$
\end{itemize}
Here,
   \begin{equation*}
      \mathcal{B}_{1}v:=\begin{cases}
          \smallskip v(0,t)\\
          \smallskip  v (L,t)  \\
            v_{x}(L,t)+a_{30}v(0,t)+b_{30}v(L,t),
        \end{cases}
    \end{equation*}

\smallskip

    \begin{equation*}
   \mathcal{B}_{2}v:=\begin{cases}
     \smallskip v(0,t) \\
      \smallskip  v_{x}(L,t)+a_{30}v(0,t)+b_{30}v(L,t) \\
\displaystyle         v_{xx}(L,t)+  \sum _{j=0}^1\left( a_{2j} \partial ^j_x v(0,t) +  b_{2j}\partial ^j_x v(L,t)\right),
        \end{cases}
    \end{equation*}

 \smallskip

     \begin{equation*}
      \mathcal{B}_{3}v:=\begin{cases}
\displaystyle       \smallskip  v_{xx}(0,t) +\sum _{j=0}^1\left( a_{1j} \partial ^j_x v(0,t)+b_{1j}\partial ^j_x v(L,t)\right),\\
      \smallskip v(L,t) \\
       v_{x}(L,t)+a_{30}v(0,t)+b_{30}v(L,t)
        \end{cases}
    \end{equation*}
and
    \begin{equation*}
  \mathcal{B}_{4}v:=\begin{cases}
\displaystyle   \smallskip v_{xx}(0,t)+ \sum _{j=0}^1\left( a_{1j} \partial ^j_x v(0,t)+b_{1j}\partial ^j_x v(L,t)\right) \\
   \smallskip  v_{x}(L,t)+a_{30}v(0,t)+b_{30}v(L,t)  \\
\displaystyle    v_{xx}(L,t)+  \sum _{j=0}^1\left( a_{2j} \partial ^j_x v(0,t) +  b_{2j}\partial ^j_x v(L,t)\right).
        \end{cases}
    \end{equation*}

\end{itemize}
As a comparison, note that the assumptions of Theorem A  are satisfied if  and only if    one of the following boundary conditions are imposed on the equation in (\ref{1.3}).

 \begin{itemize}
 \item[(a)] $u(0,t)=0, \quad u(1,t)=0, \quad u_x (1,t)=0;$

 \item[(b)] \[ \mbox{$u_{xx}(0,t)+au_x(0,t)+bu(0,t)=0, \quad u_x(1,t)=0,
 \quad u(1,t)=0$ }\]
 with
 \begin{equation}\label{z-1} a>b^2/2;\end{equation}

 \item[(c)] \[\mbox{$u(0,t)=0, \quad u_{xx}(1,t)+au_x(1,t)+bu(1,t)=0,
 \quad u_x(1,t)+cu(1,t)=0,$  }\]
 with
 \begin{equation}\label{z-2} ac>b-c^2/2;\end{equation}

 \item[(d)] $u_{xx}(0,t)+a_1u_x(0,t)+a_2 u(0,t)=0, \quad
 u_{xx}(1,t)+b_1u_x(1,t)+b_2 u(1,t)=0,$ \[ u_x(1,t)+cu(1,t)=0, \]
 with
 \begin{equation} \label{z-3} a_2 > a_1^2/2, \quad b_1c > b_2 -c^2/2 .\end{equation}
Then,  it follows of our results that the conditions (\ref{z-1}), (\ref{z-2}) and (\ref{z-3}) for Theorem A can be removed completely.
 \item[(iii)]  In Theorem \ref{th1},  we replace the $s-$compatibility of $(\phi, \vec{h}) $ (cf. Theorem C)  by assuming $(\phi, \vec{h}) \in H^s_0 (0,L)\times H^{\frac{s+1}{3}} _0(0,L] \times H^{\frac{s+1}{3}}_0 (0,L]\times H^{\frac{s}{3}}_0 (0,L]$ for simplicity.  The same remarks hold for Theorems \ref{th2}-\ref{th4} too.
\end{itemize}
}

To prove our theorems, we rewrite the boundary operators $\mathcal{B}_k,  \ k=1,2,3,4$ as
\[ \mathcal{B}_k = \mathcal{B}_{k,0}+ \mathcal{B}_{k,1} \]
with

%For the linear system associated to \eqref{1.1},
 %  \begin{equation}  \begin{cases}   v_t+v_{xxx}=0,&x \in (0,L), \ t\geq 0,\\
 %  v(x,0)=\phi(x),&x\in (0,L),   \end{cases}   \label{int_lin_1}   \end{equation}denote by $\mathcal{B}_{k,0}v$, for $k=1,2,3,4$, the following boundary conditions:
 \[
      \mathcal{B}_{1,0}v:= ( v(0,t), v (L,t), v_{x}(L,t)), \quad
\mathcal{B}_{2,0}v:=( v(0,t), v_{x}(L,t),
        v_{xx}(L,t) ),\]

\[
      \mathcal{B}_{3,0}v:=( v_{xx}(0,t), v(L,t),   v_{x}(L,t)), \quad   \mathcal{B}_{4,0}v:=( v_{xx}(0,t), v_{x}(L,t),
    v_{xx}(L,t))\]
 and
$$\mathcal{B}_{1,1}v:=\left (0,\ 0, a_{30}v(0,t)+b_{30}v(L,t)\right ),$$
$$\mathcal{B}_{2,1}v:=\left (0,\  a_{30}v(0,t)+b_{30}v(L,t),\ \sum _{j=0}^1\left( a_{2j} \partial ^j_x v(0,t) + b_{2j}\partial ^j_x v(L,t)\right)\right),$$
$$\mathcal{B}_{3,1}v:=\left(\sum _{j=0}^1 \left(a_{1j} \partial ^j_x v(0,t)+b_{1j}
\partial ^j_x v(L,t)\right),\ 0,\ a_{30}v(0,t)+b_{30}v(L,t)\right),$$

\begin{align*}
\mathcal{B}_{4,1}v:=\left(\sum _{j=0}^1 \left(a_{1j} \partial ^j_x v(0,t)+b_{1j}
\partial ^j_x v(L,t)\right),\  a_{30}v(0,t)+b_{30}v(L,t)\right.,\\
\left . \sum _{j=0}^1 \left(a_{2j} \partial ^j_x v(0,t) + b_{2j}\partial ^j_x v(L,t)\right)\right).
\end{align*}
To prove our main result, we will first study the linear IBVP
\begin{equation}\label{y-1}
\begin{cases}
u_t +u_{xxx} +\delta_k u=f, \quad x\in (0,L), \quad t >0\\ u(x,0)= \phi (x), \\ \mathcal{B}_{k,0} u= \vec{h},
\end{cases}
\end{equation}
to establish all the linear estimates needed later for  dealing with the nonlinear IBVP (\ref{1.1})-(\ref{1.2}).  Here $\delta _k=0$ for $k=1,2,3$ and $\delta _4=1$.
 Then we  will consider the  nonlinear map $\Gamma $ defined by the following IBVP
\begin{equation}\label{y-2}
\begin{cases}
u_t +u_{xxx} +\delta_ku= -v_x -vv_x +\delta_kv , \quad x\in (0,L), \quad t >0\\ u(x,0)= \phi (x), \\ \mathcal{B}_{k,0} u= \vec{h}-\mathcal{B}_{k,1} v
\end{cases}
\end{equation}
with
\[ \Gamma (v)=u.\]
We will  show that $\Gamma$ is a contraction in an appropriated space whose  fixed point will be the desired solution of the nonlinear  IBVP (\ref{1.1})-(\ref{1.2}).
The key to show that $\Gamma$ is a contraction in an appropriate space is the sharp Kato smoothing  property of the solution of the IBVP (\ref{y-1}) as described below, for example, for $s=0$:

\smallskip
 {\it For given $\phi \in  L^2 (0,L)$ and $f\in L^1 (0,T; L^2 (0,L))$ and $\vec{h}\in {\cal H}^0_k (0,T)$,  the IBVP (\ref{y-1}) admits a unique solution $u\in C([0,T]; L^2 (0,L))\cap L^2 (0,T; H^1 (0,L))$
with
$$\partial_x^lu\in L^{\infty}(0,L;H^{\frac{1-l}{3}}(0,T)) \text{ for }l=0,1,2.$$
}
%\smallskip

In order to demonstrate  the sharp Kato smoothing properties for solutions of the IBVP (\ref{y-1}), we need  to study the following IBVP
\begin{equation}\label{y-3}
\begin{cases}
u_t +u_{xxx}+\delta _ku=0, \quad x\in (0,L), \quad t >0\\ u(x,0)= 0, \\ \mathcal{B}_{k,0} u= \vec{h}
\end{cases}
\end{equation}
for $k=1,2,3,4$.
The corresponding solution map $$\vec{h} \to u$$ will be called the \textit{boundary integral operator} denoted by ${\cal W}_{bdr} ^{(k)}$.  An explicit representation  formula will be given for this boundary integral operator that will play important role  in showing the solution of the IBVP (\ref{y-3}) possesses the sharp Kato smoothing properties. The needed  sharp Kato smoothing properties for solutions of the IBVP (\ref{y-1}) will then follow from the sharp Kato smoothing properties for solutions of the IBVP (\ref{y-3}) and  the well-known sharp Kato smoothing properties for solutions of the Cauchy problem
\[ u_t +u_{xxx} +\delta _ku=0, \quad u(x,0)=\psi (x), \quad x, \ t\in \mathbb{R}.\]

\smallskip
The plan of the present paper is as follows.

\smallskip
--- In Section \ref{sec2}  we will study the linear IBVP (\ref{y-1}) The explicit representation formulas for the boundary integral operators ${\cal W}_{bdr} ^{(k)}$, for $k=1,2,3,4$, will be first presented. The various linear  estimates for solutions of the IBVP (\ref{y-1}) will be derived including the sharp Kato smoothing properties.

\smallskip
--- The Section \ref{sec3} is devoted to well-posedness of the nonlinear problem \eqref{1.1}-\eqref{1.2} will be established.

\smallskip
--- Finally, in the Section \ref{sec4}, some  conclusion remarks will be presented together with some open problems  for further investigations.

    \section{Linear problems}\label{sec2}
\setcounter{equation}{0}
This section  is devoted to study the linear IBVP  (\ref{y-1}) which will  be divided into two subsections. In  subsection \ref{sec21}, we will  present an explicit representation for the boundary integral operators ${\cal W}_{bdr} ^{(k)}$  and then solution formulas for   the solutions of the IBVP (\ref{y-1}). Various linear estimates for solutions of the IBVP (\ref{y-1}) will be derived in subsection \ref{sec22}.

\subsection{Boundary integral operators and their applications}\label{sec21}
\setcounter{equation}{0}
In this subsection, we first  derive explicit representation formulas for   the following four classes of
nonhomogeneous boundary-value problems
   \begin{equation}\label{2.1-1}
       \begin{cases}
            v_t+ v_{xxx}=0, \quad v(x,0)=0, &x \in (0,L), \  t\geq 0,\\
             \mathcal{B}_{1,0}v=(h_{1,1}(t),\ h_{2,1}(t),\ h_{3,1}(t)), & t\geq 0,
        \end{cases}
    \end{equation}

    \begin{equation}\label{2.1-2}
       \begin{cases}
            v_t+ v_{xxx}=0, \quad v(x,0)=0, &x \in (0,L), \  t\geq 0,\\
           \mathcal{B}_{2,0}v=(h_{1,2}(t),\ h_{2,2}(t),\ h_{3,2}(t)),  & t\geq 0,
        \end{cases}
    \end{equation}

    \begin{equation}\label{2.1-3}
       \begin{cases}
            v_t+ v_{xxx}=0, \quad v(x,0)=0, &x \in (0,L), \  t\geq 0,\\
             \mathcal{B}_{3,0}v=(h_{1,3}(t),\ h_{2,3}(t),\ h_{3,3}(t)),  & t\geq 0
                     \end{cases}
    \end{equation}
    and
    \begin{equation}\label{2.1-4}
       \begin{cases}
            v_t+ v_{xxx} +v=0, \quad v(x,0)=0, &x \in (0,L), \  t\geq 0,\\
           \mathcal{B}_{4,0}v=(h_{1,4}(t),\ h_{2,4}(t),\ h_{3,4}(t)),  & t\geq 0.
        \end{cases}
    \end{equation}
%with $\mathcal{B}_{k,0}$, $k=1,2,3,4$, defined by \eqref{bc_a}, \eqref{bc_b}, \eqref{bc_c} and \eqref{bc_d}, respectively.
%We  derive  explicit solution formulas for these nonhomogeneousboundary value problems.
Without loss of generality, we assume
that $L=1$ in this subsection.
\medskip

Consideration is first given to the IBVP (\ref{2.1-1}). Applying
the Laplace transform with respect to $t$, (\ref{2.1-1}) is
converted to
    \begin{equation}\label{2.2-1}
        \begin{cases}
            s\h{v}+  \h{v}_{xxx}=0,\\
            \h{v}(0,s)=\h{h}_{1,1}(s),\  \h{v}  (1,s)=\h{h}_{2,1}(s),\
            \h{v}_{x}(1,\xi)=\h{h}_{3,1}(s),
        \end{cases}
    \end{equation}
    where  $$\h{v}(x,\xi)=\int_0^{+\infty} e^{-s t}v(x,t)dt$$
    and
    \[  \hat{h}_j (s) =  \int^ {\infty}_0 e^{-st} h_{j,1}(t) dt, \quad j=1,2,3.\]
The solution of (\ref{2.2-1}) can be written in the form
\[ \h{v}(x,s)=\sum_{j=1}^3 c_j(s) e^{\lam_j(s)x},\]
where $\lambda _j (s), j=1,2,3$ are solutions of the characteristic
equation
\[ s   + \lam^3=0\]
and $c_j (s), j=1,2,3$, solves the linear system
%\begin{equation}\label{2.3-1}
\[ \underbrace{\begin{pmatrix}
1 &  1& 1\\
 e^{\lam_1} &  e^{\lam_2} &  e^{\lam_3}\\
\lam_1 e^{\lam_1} & \lam_2 e^{\lam_2} & \lam_3  e^{\lam_3}\\
\end{pmatrix}}_{A^1}
\begin{pmatrix}
c_1\\ c_2\\c_3
\end{pmatrix}
= \underbrace{\begin{pmatrix} \h{h}_{1,1}\\ \h{h}_{2,1} \\
\h{h}_{3,1}
\end{pmatrix}}_{\vec{\widehat{h}}_1}.
%\end{equation}
\]
By Cramer's rule,
$$c_j= \frac{\Delta_j^1(s)}{\Delta^1(s)}, \ j=1,2,3,$$
 with $\Delta ^1$ the determinant of $A^1$ and $\Delta_j^1$ the determinant of the
 matrix
 $A^1$ with the column $j$ replaced by ${\vec{\widehat{h}}_1}$.
 Taking the inverse Laplace transform of $\widehat{v}$ and following the same arguments as
 that in \cite{BSZ03FiniteDomain} yield  the representation
 \[ v(x,t)=\sum ^3_{m=1} v_m^1 (x,t)\]
 with
 \[ v_m^1 (x,t)=\sum ^3_{j=1}v_{j,m}^1(x,t)\]
 and
 \[ v_{j,m}^1(x,t)= v_{j,m}^{+,1} (x,t)+v_{j,m}^{-,1}(x,t)\]
 where
 \[ v_{j,m}^{+,1}(x,t)=\frac{1}{2\pi i}  \int_{0}^{+i\infty}  e^{s t}
 \frac{\Delta_{j,m}^1(s)}{\Delta ^1(s)}  \h{h}_{m,1}(s) e^{\lam_j(s)x}ds\]
 and
 \[ v_{j,m}^{-,1}(x,t)=\frac{1}{2\pi i}  \int_{-i\infty}^{0}  e^{s t}
 \frac{\Delta_{j,m}^1(s)}{\Delta ^1(s)}  \h{h}_{m,1}(s) e^{\lam_j(s)x}ds,\]
 for $j, m =1,2,3$.   Here $\Delta_{j,m}^1(s)$ is obtained from $\Delta _j^1(s)$ by letting $\hat{h}_m (s) = 1$ and
$\hat{h}_k(s) = 0$ for $k \ne m, $ $k,m = 1, 2, 3$. More
precisely,
\[ \Delta ^1= (\lambda _3-\lambda _2)e^{-\lambda _1}+
(\lambda_1-\lambda _3)e^{-\lambda _2}+ (\lambda _2-\lambda
_1)e^{-\lambda _3} ;\]
\[ \Delta _{1,1}^1= (\lambda _3-\lambda _2)e^{-\lambda _1}, \ \Delta _{2,1}^1= (\lambda _1-\lambda _3)e^{-\lambda
_2}, \ \Delta _{3,1}^1= (\lambda _2-\lambda _1)e^{-\lambda _3};\]
\[ \Delta ^1_{1,2} =\lambda _2e^{\lambda _2}-\lambda _3e^{\lambda
_3}, \ \Delta ^1_{2,2} =\lambda _3e^{\lambda _3}-\lambda
_1e^{\lambda _1}, \ \Delta ^1_{3,2} =\lambda _1e^{\lambda
_1}-\lambda _2e^{\lambda _2};\]
\[ \Delta _{1,3}^1= e^{\lambda _3}-e^{\lambda _2}, \  \Delta _{2,3}^1= e^{\lambda _1}-e^{\lambda
_3}, \  \Delta _{3,3}^1= e^{\lambda _2}-e^{\lambda _1}.\]
   Making the substitution $s=i \rho ^3  $, with $0\leq \rho < \infty ,$ in the  the characteristic equation
   $$  s   + \lam^3=0,$$
   the three roots $\lambda _j$, $j=1,2,3$, are

   \[ \lambda _1^+ (\rho) =i\rho, \quad \lambda ^+_2
(\rho)=\frac{\sqrt{3}-i}{2} \rho, \quad \lambda ^+_3 (\rho) =
-\frac{\sqrt{3}+i}{2} \rho. \]
Thus $v_{j,m}^{+,1} (x,t)$ has the form

\[ v_{j,m}^{+,1} (x,t)= \ds\frac{1}{2\pi}   \int^{\infty}_ 0 e^{i \rho^3  t}
    \frac{\Delta_{j,m}^{+,1}(\rho)}{\Delta ^{+,1}(\rho)} \h{h}_{m,1}^+(\rho) e^{\lam_j^+(\rho)x}
     3\rho^2 d\rho\]
    and \[ v_{j,m}^{-,1}(x,t)= \overline{v_{j,m}^{+,1} (x,t)},\]
    where $\h{h}_{m,1}^+(\rho)=\h{h}_m(i \rho ^3 )$, $\Delta
    ^{+,1}(\rho)$ and $\Delta ^{+,1}_{j,m} (\rho)$ are obtained from
    $\Delta ^1
    (s)$ and $\Delta _{j,m}^1 (s)$ by replacing $s$ with $i \rho
    ^3 $ and  $\lambda _j ^+ (\rho)=\lambda _j (i \rho
    ^3 )$.

For given $m, j=1,2,3$,  let $W_{j,m}^1 $ be  an  operator on
$H^s_0 (\mathbb{R}^+)$ defined as follows:  for any $h\in H^s _0
(\mathbb{R}^+)$,
\begin{equation}
 [W_{j,m}^1h](x,t) \equiv [U_{j,m}^1 h](x,t) +\overline{[U_{j,m}^1h](x,t)}\label{2.2-1a}
\end{equation}
with
\begin{equation} [U_{j,m}^1 h](x,t)\equiv \frac{1}{2\pi } \int ^{+\infty }_{0} e^{i\rho ^3
t} e^{\lambda ^+_j (\rho ) x}   3\rho ^2
[Q_{j,m}^{+,1}h](\rho )d\rho \label{2.3-1a}
\end{equation}
for $j=1,3, \ m=1,2,3 $ and
\begin{equation}
 [U_{2,m}^1 h](x,t) \equiv  \frac{1}{2\pi } \int ^{+\infty }_{0} e^{i \rho ^3
 t} e^{-(\lambda ^+_2 (\rho ) (1-x)}    3\rho ^2
[Q_{2,m}^{+,1}h](\rho )d\rho \label{2.4-1}
\end{equation}
for $m=1,2,3$. Here
\begin{equation}\label{2.5-1} [Q_{j,m}
^{+,1}h] (\rho ):=\frac{\Delta ^{+,1}_{j,m} (\rho ) } {\Delta
^{+,1} (\rho )} \hat{h}   ^+ (\rho ), \qquad [Q_{2,m}^{+,1}h]
(\rho )=\frac{\Delta ^{+,1}_{2,m} (\rho ) } {\Delta ^{+,1} (\rho
)} e^{\lambda ^+_2 (\rho )} \hat{h} ^+ (\rho )\end{equation} for
$j=1,3 $ and $m=1,2,3$, $\hat{h}^+ (\rho) = \hat{h} (i\rho ^3)$.
Then the solution of the IBVP (\ref{2.1-1}) has the following
representation.

\begin{lem}\label{lem1a}
Given  $\vec{h}_1=(h_{1,1},h_{2,1},h_{3,1})$, the  solution $v$ of
the IBVP (\ref{2.1-1}) can be written in the form
\[
v(x,t)=[W_{bdr}^1\vec{h}_1](x,t):= \sum_{j,m=1}^3
[W_{j,m}^1h_{m,1}](x,t).
\]
\end{lem}

Next we consider the IBVP (\ref{2.1-2}). A similarly arguments
shows the solution of the IBVP (\ref{2.1-2}) has the following
representation.
\begin{lem}
The solution $v$ of the IBVP (\ref{2.1-2}) can be written in the
form
\[
v(x,t)=[W_{bdr}^2\vec{h}_2](x,t):= \sum_{j,m=1}^3
[W_{j,m}^2h_{m,2}](x,t),
\]
where
\begin{equation}
 [W_{j,m}^2h](x,t) \equiv [U_{j,m}^2 h](x,t) +\overline{[U_{j,m}^2h](x,t)}\label{2.2-2}
\end{equation}
with
\begin{equation} [U_{j,m}^2 h](x,t)\equiv \frac{1}{2\pi } \int ^{+\infty }_{0} e^{i\rho ^3
t} e^{\lambda ^+_j (\rho ) x}   3\rho ^2
[Q_{j,m}^{+,2}h](\rho )d\rho \label{2.3-2}
\end{equation}
for $j=1,3, \ m=1,2,3 $ and
\begin{equation}
 [U_{2,m}^2 h](x,t) \equiv  \frac{1}{2\pi } \int ^{+\infty }_{0} e^{i \rho ^3
 t} e^{-\lambda ^+_2 (\rho ) (1-x)}    3\rho ^2
[Q_{2,m}^{+,2}h](\rho )d\rho \label{2.4-2}
\end{equation}
for $m=1,2,3$. Here
 \begin{equation}\label{2.5-2} [Q_{j,m}
^{+,2}h] (\rho ):=\frac{\Delta ^{+,2}_{j,m} (\rho ) } {\Delta
^{+,2} (\rho )} \hat{h}   ^+ (\rho ), \qquad [Q_{2,m}^{+,2}h]
(\rho )=\frac{\Delta ^{+,2}_{2,m} (\rho ) } {\Delta ^{+,2} (\rho
)} e^{\lambda ^+_2 (\rho )} \hat{h} ^+ (\rho )\end{equation} for
$j=1,3 $ and $m=1,2,3$. Here $\hat{h}^+ (\rho) = \hat{h} (i\rho
^3)$, $\Delta
    ^{+,2}(\rho)$ and $\Delta ^{+,2}_{j,m} (\rho)$ are obtained from
    $\Delta ^2
    (s)$ and $\Delta _{j,m}^2 (s)$ by replacing $s$ with $i \rho
    ^3$ and  $\lambda _j ^+ (\rho)=\lambda _j (i \rho
    ^3 )$ where
    \[ \Delta ^2 = \lambda _2 \lambda _3(\lambda _3-\lambda _2) e^{-\lambda _1}
    +\lambda _1 \lambda _3(\lambda _1-\lambda _3) e^{-\lambda _2}+\lambda _2 \lambda _1(\lambda _2-\lambda _1) e^{-\lambda
    _3};\]
    \[ \Delta ^2_{1,1} = e^{-\lambda _1}\lambda _2 \lambda _3(\lambda _3-\lambda
    _2), \  \Delta ^2_{2,1} = e^{-\lambda _2}\lambda _1 \lambda _3(\lambda
    _3-\lambda _1), \  \Delta ^2_{3,1} = e^{-\lambda _3}\lambda _1 \lambda _2(\lambda
    _2-\lambda _1);\]
    \[ \Delta ^2_{1,2} = \lambda ^2_2 e^{\lambda _2} -\lambda _3^2
    e^{\lambda _3}, \ \Delta ^2_{2,2} = \lambda ^2_3 e^{\lambda _3} -\lambda _1^2
    e^{\lambda _1}, \ \Delta ^2_{3,2} = \lambda_1 ^2 e^{\lambda _1} -\lambda_ 2^2
    e^{\lambda _2};\]
    \[ \Delta ^2_{1,3} = \lambda _3 e^{\lambda _3}-\lambda _2
    e^{\lambda _2}, \ \Delta ^2_{2,3} = \lambda _1 e^{\lambda _1}-\lambda
    _3
    e^{\lambda _3}, \ \Delta ^2_{3,3} = \lambda _2 e^{\lambda _2}-\lambda
    _1
    e^{\lambda _1}.\]
\end{lem}

\medskip
For solutions of (\ref{2.1-3}), we have the following lemma.
\begin{lem}\label{l3-c}
The solution $v$ of the IBVP (\ref{2.1-3}) can be written in the
form
\[
v(x,t)=[W_{bdr}^3\vec{h}_3](x,t):= \sum_{j,m=1}^3
[W_{j,m}^3h_{m,3}](x,t),
\]
where
\begin{equation}
 [W_{j,m}^3h](x,t) \equiv [U_{j,m}^3 h](x,t) +\overline{[U_{j,m}^3h](x,t)}\label{2.2-3}
\end{equation}
with
\begin{equation} [U_{j,m}^3 h](x,t)\equiv \frac{1}{2\pi } \int ^{+\infty }_{0} e^{i\rho ^3
t} e^{\lambda ^+_j (\rho ) x}   3\rho ^2
[Q_{j,m}^{+,3}h](\rho )d\rho \label{2.3-3}
\end{equation}
for $j=1,3, \ m=1,2,3 $ and
\begin{equation}
 [U_{2,m}^3 h](x,t) \equiv  \frac{1}{2\pi } \int ^{+\infty }_{0} e^{i \rho ^3
 t} e^{-\lambda ^+_2 (\rho )(1- x)}    3\rho ^2
[Q_{2,m}^{+,3}h](\rho )d\rho \label{2.4-3}
\end{equation}
for $m=1,2,3$.
Here
\begin{equation}\label{2.5-3} [Q_{j,m}
^{+,3}h] (\rho ):=\frac{\Delta ^{+,3}_{j,m} (\rho ) } {\Delta
^{+,3} (\rho )} \hat{h}   ^+ (\rho ), \qquad [Q_{2,m}^{+,3}h]
(\rho )=\frac{\Delta ^{+,3}_{2,m} (\rho ) } {\Delta ^{+,3} (\rho
)} e^{\lambda ^+_2 (\rho )} \hat{h} ^+ (\rho )\end{equation} for
$j=1,3 $ and $m=1,2,3$. Here $\hat{h}^+ (\rho) = \hat{h} (i\rho
^3)$, $\Delta
    ^{+,1}(\rho)$ and $\Delta ^{+,3}_{j,m} (\rho)$ are obtained from
    $\Delta ^3
    (s)$ and $\Delta _{j,m}^3 (s)$ by replacing $s$ with $i \rho
    ^3$ and  $\lambda _j ^+ (\rho)=\lambda _j (i \rho
    ^3 )$ where
    \[ \Delta ^3 = \lambda _1^2(\lambda _3-\lambda _2) e^{-\lambda _1}
    +\lambda _2^2(\lambda _1-\lambda _3) e^{-\lambda _2}+\lambda _3^2(\lambda _2-\lambda _1) e^{-\lambda
    _3};\]
    \[ \Delta ^3_{1,1} = e^{-\lambda _1}(\lambda _3-\lambda
    _2), \  \Delta ^3_{2,1} = e^{-\lambda _2}(\lambda
    _1-\lambda _3), \  \Delta ^3_{3,1} = e^{-\lambda _3}(\lambda
    _2-\lambda _1);\]
    \[ \Delta ^3_{1,2} = \lambda ^2\lambda _3(\lambda _3 e^{\lambda _2} -\lambda
    _2
    e^{\lambda _3}), \ \Delta ^3_{2,2} = \lambda _1\lambda _3 (\lambda _1e^{\lambda _3} -\lambda
    _3
    e^{\lambda _1}), \ \Delta ^3_{3,2} = \lambda_1 \lambda _2( \lambda _2e^{\lambda _1}
    -\lambda _1
    e^{\lambda _2});\]
    \[ \Delta ^3_{1,3} = \lambda _2^2 e^{\lambda _3}-\lambda _3^2
    e^{\lambda _2}, \ \Delta ^3_{2,3} = \lambda _3^2 e^{\lambda _1}-\lambda
    _1^2
    e^{\lambda _3}, \ \Delta ^3_{3,3} = \lambda _1^2e^{\lambda _2}-\lambda
    _2^2
    e^{\lambda _1}.\]
\end{lem}

\medskip
For  solutions of (\ref{2.1-4}), we have
\begin{lem}\label{l4-d}
The solution $v$ of the IBVP (\ref{2.1-4}) can be written in the
form
\[
v(x,t)=[W_{bdr}^4\vec{h}_4](x,t):= \sum_{j,m=1}^3
[W_{j,m}^4h_{m,4}](x,t),
\]
where
\begin{equation}
 [W_{j,m}^4h](x,t) \equiv [U_{j,m}^4 h](x,t) +\overline{[U_{j,m}^4h](x,t)}\label{2.2-4}
\end{equation}
with
\begin{equation} [U_{j,m}^4 h](x,t)\equiv \frac{1}{2\pi } \int ^{+\infty }_{0} e^{i\rho ^3
t} e^{\lambda ^+_j (\rho ) x}   3\rho ^2
[Q_{j,m}^{+,4}h](\rho )d\rho \label{2.3-4}
\end{equation}
for $j=1,3, \ m=1,2,3 $ and
\begin{equation}
 [U_{2,m}^4 h](x,t) \equiv  \frac{1}{2\pi } \int ^{+\infty }_{0} e^{i \rho ^3
 t} e^{-\lambda ^+_2 (\rho ) (1-x)}    3\rho ^2
[Q_{2,m}^{+,4}h](\rho )d\rho \label{2.4-4}
\end{equation}
for $m=1,2,3$.
Here
\begin{equation}\label{2.5-4} [Q_{j,m}
^{+,4}h] (\rho ):=\frac{\Delta ^{+,4}_{j,m} (\rho ) } {\Delta
^{+,4} (\rho )} \hat{h}   ^+ (\rho ), \qquad [Q_{2,m}^{+,4}h]
(\rho )=\frac{\Delta ^{+,4}_{2,m} (\rho ) } {\Delta ^{+,4} (\rho
)} e^{\lambda ^+_2 (\rho )} \hat{h} ^+ (\rho )\end{equation} for
$j=1,3 $ and $m=1,2,3$. Here $\hat{h}^+ (\rho) = \hat{h} (i\rho
^3)$, $\Delta
    ^{+,4}(\rho)$ and $\Delta ^{+,4}_{j,m} (\rho)$ are obtained from
    $\Delta ^4
    (s)$ and $\Delta _{j,m}^4 (s)$ by replacing $s$ with $i \rho
    ^3  $ and  $\lambda _j ^+ (\rho)=\lambda _j (i \rho
    ^3 )$ where
    \[ \Delta ^4 = \lambda _1\lambda _2 \lambda _3\left (\lambda _1(\lambda _3-\lambda _2) e^{-\lambda _1}
    +\lambda _2(\lambda _1-\lambda _3) e^{-\lambda _2}+\lambda _3(\lambda _2-\lambda _1) e^{-\lambda
    _3}\right );\]
    \[ \Delta ^4_{1,1} = e^{-\lambda _1}\lambda _2 \lambda _3(\lambda _3-\lambda
    _2), \  \Delta ^4_{2,1} = e^{-\lambda _2}\lambda _1 \lambda _3(\lambda
    _1-\lambda _3), \  \Delta ^4_{3,1} = e^{-\lambda _3}\lambda _1 \lambda _2(\lambda
    _2-\lambda _1);\]
    \[ \Delta ^4_{1,2} = \lambda ^2_2 \lambda _3^2 (e^{\lambda _2} -
    e^{\lambda _3}), \ \Delta ^4_{2,2} = \lambda _1^2\lambda ^2_3( e^{\lambda _3} -
    e^{\lambda _1}), \ \Delta ^4_{3,2} = \lambda_1 ^2\lambda _2^2 ( e^{\lambda _1} -
    e^{\lambda _2});\]
    \[ \Delta ^4_{1,3} = \lambda _2\lambda _3 (\lambda _2e^{\lambda _3}-\lambda
    _3
    e^{\lambda _2}), \ \Delta ^4_{2,3} = \lambda _1 \lambda _3 (\lambda_3e^{\lambda _1}-\lambda
    _1
    e^{\lambda _3}), \ \Delta ^4_{3,3} = \lambda _1 \lambda _2 (\lambda _1e^{\lambda _2}-\lambda
    _2
    e^{\lambda _1}),\]
    and  $\lambda _j (s), j=1,2,3$ are solutions of the characteristic
equation
\[ s +1  + \lam^3=0.\]
\end{lem}

\begin{rem}
From  $s\hat v + \hat v_{xxx}= 0$ with boundary conditions ${\cal B}_j \hat v = 0 $ for $j = 1,2,3$ or  $s\hat v + \hat v + \hat v_{xxx}= 0$ with boundary conditions ${\cal B}_4 \hat v = 0 $, it can be easily shown that there are no nontrivial solutions $\hat v$ for any $s$ with $\emph{Re\,} s \geq 0$. Therefore, $\Delta ^j (s) \not = 0 , j = 1, 2, 3, 4$ for any $s$ with $\emph{Re\,} s \geq 0$.
\end{rem}

%\begin{rem}
%If $\rho=0$, observe that $\lambda_1(0)=\lambda_2(0)=\lambda_3(0)=0$, thus $\Delta^4(\rho)=0$, for $\rho=0$. Therefore, the integral
%$$w^{+,4}_{j,m}(x,t)=\frac{1}{2\pi i}\int_0^{+\infty}e^{i\rho^3t}\frac{\Delta^{+,4}_{j,m}(\rho)}{\Delta^{+,4}(\rho)}\hat{h}^+_{m,4}(\rho)e^{\lambda^+_j(\rho)x}3i\rho^2d\rho,$$
%can be divergent. To fix this issue, we will divide the $w^4_{j,m}(x,t)$ into three pieces as follows
%\begin{align*}
%w^4_m(x,t)&=\sum_{j=1}^{3}\frac{1}{2\pi i}\left(\int^{+i\infty}_{i}\frac{\Delta^4_{j,m}(s)}{\Delta^4(s)}e^{\lambda_j(s)x}\hat{h}_{m,4}(s)ds+\int^{-i}_{-i\infty}\frac{\Delta^4_{j,m}(s)}{\Delta^4(s)}e^{\lambda_j(s)x}\hat{h}_{m,4}(s)ds\right)\\
%&+\sum_{j=1}^{3}\frac{1}{2\pi i}\int_{\Gamma}\frac{\Delta^4_{j,m}(s)}{\Delta^4(s)}e^{\lambda_j(s)x}\hat{h}_{m,4}(s)ds\\
%&:=w_{j,m}^{+,4}(x,t)+w_{j,m}^{-,4}(x,t)+\sum_{j=1}^{3}\frac{1}{2\pi i}\int_{\Gamma}\frac{\Delta^4_{j,m}(s)}{\Delta^4(s)}e^{\lambda_j(s)x}\hat{h}_{m,4}(s)ds.
%\end{align*}
%Here, $\Gamma:=\partial\gamma$, where $\gamma=\{r[cos(t)+isin(t)]: -\pi<t<\pi $ and $ r>0 \}$. Therefore, the behavior of $w^{+,4}_{j,m}(x,t)+w^{-,4}_{j,m}(x,t)$ is the same as above. Moreover, the quantity $\frac{\Delta^4_{j,m}(s)}{\Delta^4(s)}$ is continuous and has not singularities due to choose of $\gamma$. For the other boundary conditions cases, the same argument can be done.
%\end{rem}

The following lemma is helpful in deriving various linear estimates for solutions of the  IBVP (\ref{y-1})  in  the  next subsection.

\begin{lem}\label{l5-abcd}
For $m=1,2,3$ , $k=1,2,3,4$ and $j=1,3$, set
$$\h{h}_{j,m,k}^{*}(\rho):= 3\rho ^2 [Q_{j,m}^{+,k}h_{m,k}](\rho) = 3\rho^2\frac{\Delta_{j,m}^{+,k}(\rho)}{\Delta^{+,k}(\rho)}\h{h}_{m,k}^+(\rho)
$$
 and
 \[   \h{h}_{2,m,k}^{*}(\rho):=3\rho ^2 [Q_{2,m}^{+,k}h_{m,k}](\rho) =3\rho ^2  \frac{\Delta_{2,m}^{+,k}(\rho)}{\Delta
^{+,k}(\rho)}e^{\lam_2^+(\rho)}\h{h}^+_{m,k}(\rho)\] and view
$h^*_{j,m,k}$ as the inverse Fourier transform of
$\h{h}_{j,m,k}^*$. Then for any $s\in \mathbb{R}$,
\begin{equation}\label{2.6-1} \left \{
\begin{array}{lll} h_{1,1}\in H_0^{(s+1)/3}(\mathbb{R}^+) &\Rightarrow h_{j,1,1}^* \in
H^s (\mathbb{R}), & \ j=1,2,3,\\ \\
  h_{2,1}\in H_0^{(s+1)/3}(\mathbb{R}^+) &\Rightarrow h^*_{j,2,1}\in
H^s (\mathbb{R}),& \ j=1,2,3,\\ \\
 h_{3,1}\in H_0^{s/3}(\mathbb{R}^+)&\Rightarrow h_{j,3,1}^* \in H^s
 (\mathbb{R}), & \ j=1,2,3.
 \end{array} \right .
 \end{equation}
\begin{equation}\label{2.6-2} \left \{
\begin{array}{lll} h_{1,2}\in H_0^{(s+1)/3}(\mathbb{R}^+) &\Rightarrow h_{j,1,2}^* \in
H^s (\mathbb{R}), & \ j=1,2,3,\\ \\
  h_{2,2}\in H_0^{s/3}(\mathbb{R}^+) &\Rightarrow h^*_{j,2,2}\in
H^s (\mathbb{R}),& \ j=1,2,3,\\ \\
 h_{3,2}\in H_0^{(s-1)/3}(\mathbb{R}^+)&\Rightarrow h_{j,3,2}^* \in H^s
 (\mathbb{R}), & \ j=1,2,3.
 \end{array} \right .
 \end{equation}

 \begin{equation}\label{2.6-3} \left \{
\begin{array}{lll} h_{1,3}\in H_0^{(s-1)/3}(\mathbb{R}^+) &\Rightarrow h_{j,1,3}^* \in
H^s (\mathbb{R}), & \ j=1,2,3,\\ \\
  h_{2,3}\in H_0^{(s+1)/3}(\mathbb{R}^+) &\Rightarrow h^*_{j,2,3}\in
H^s (\mathbb{R}),& \ j=1,2,3,\\ \\
 h_{3,3}\in H_0^{s/3}(\mathbb{R}^+)&\Rightarrow h_{j,3,3}^* \in H^s
 (\mathbb{R}), & \ j=1,2,3.
 \end{array} \right .
 \end{equation}

 \begin{equation}\label{2.6-4} \left \{
\begin{array}{lll} h_{1,4}\in H_0^{(s-1)/3}(\mathbb{R}^+) &\Rightarrow h_{j,1,4}^* \in
H^s (\mathbb{R}), & \ j=1,2,3,\\ \\
  h_{2,4}\in H_0^{s/3}(\mathbb{R}^+) &\Rightarrow h^*_{j,2,4}\in
H^s (\mathbb{R}),& \ j=1,2,3,\\ \\
 h_{3,4}\in H_0^{(s-1)/3}(\mathbb{R}^+)&\Rightarrow h_{j,3,4}^* \in H^s
 (\mathbb{R}), & \ j=1,2,3.
 \end{array} \right .
 \end{equation}

\end{lem}
\noindent {\bf Proof:} Recall that for $k=1,2,3$, we have
\[ \lambda _1^+ (\rho) =i\rho, \quad \lambda ^+_2
(\rho)=\frac{\sqrt{3}-i}{2} \rho, \quad \lambda ^+_3 (\rho) =
-\frac{\sqrt{3}+i}{2} \rho \] for $\rho\geq 0$,
and for $k=4$,
\[ \lambda _1^+ (\rho) \sim i\rho, \quad \lambda ^+_2
(\rho)\sim \frac{\sqrt{3}-i}{2} \rho, \quad \lambda ^+_3 (\rho) \sim
-\frac{\sqrt{3}+i}{2} \rho \] as $\rho\to +\infty$. Thus, the following asymptotic estimates of
$\frac{\Delta^{+,k}_{n,m}(\rho)}{\Delta ^{+,k}(\rho)}$, for
$m,n=1,2,3, \ k=1,2,3,4$, as $\rho \to +\infty$, hold:
\begin{center}
\renewcommand{\arraystretch}{2}
\begin{tabular}{||>{$}c<{$}| >{$}c<{$}|>{$}c<{$}||}\hline
\frac{\Delta_{1,1}^{+,1}(\rho)}{\Delta ^{+,1}(\rho)}  \sim
e^{-\frac{\sqrt{3}}{2}\rho} &
\frac{\Delta_{2,1}^{+,1}(\rho)}{\Delta ^{+,1}(\rho)}
 \sim  e^{- \sqrt{3} \rho} & \frac{\Delta_{3,1}^{+,1}(\rho)}{\Delta ^{+,1}(\rho)}  \sim 1 \\
\hline \frac{\Delta_{1,2}^{+,1}(\rho)}{\Delta ^{+,1}(\rho)} \sim 1
& \frac{\Delta_{2,2}^{+,1}(\rho)}{\Delta ^{+,1}(\rho)} \sim
 e^{-\frac{\sqrt{3}}{2}\rho} & \frac{\Delta_{3,2}^{+,1}(\rho)}{\Delta ^{+,1}(\rho)}  \sim  1 \\
 \hline
\frac{\Delta_{1,3}^{+,1}(\rho)}{\Delta ^{+,1}(\rho)}  \sim
\rho^{-1} &\frac{\Delta_{2,3}^{+,1}(\rho)}{\Delta ^{+,1}(\rho)}
\sim
 \rho^{-1}e^{-\frac{\sqrt{3}}{2}\rho}  & \frac{\Delta_{3,3}^{+,1}(\rho)}{\Delta ^{+,1}(\rho)}  \sim  \rho^{-1}
   \\ \hline
\end{tabular}
\end{center}

\begin{center}
\renewcommand{\arraystretch}{2}
\begin{tabular}{||>{$}c<{$}| >{$}c<{$}|>{$}c<{$}||}\hline
\frac{\Delta_{1,1}^{+,2}(\rho)}{\Delta ^{+,2}(\rho)}  \sim
e^{-\frac{\sqrt{3}}{2}\rho} &
\frac{\Delta_{2,1}^{+,2}(\rho)}{\Delta ^{+,2}(\rho)}
 \sim  \rho ^{-3}e^{- \sqrt{3} \rho} & \frac{\Delta_{3,1}^{+,2}(\rho)}{\Delta ^{+,2}(\rho)}  \sim 1 \\
\hline \frac{\Delta_{1,2}^{+,2}(\rho)}{\Delta ^{+,2}(\rho)} \sim
\rho ^{-1} & \frac{\Delta_{2,2}^{+,2}(\rho)}{\Delta ^{+,2}(\rho)}
\sim
 \rho ^{-1}e^{-\frac{\sqrt{3}}{2}\rho} & \frac{\Delta_{3,2}^{+,2}(\rho)}{\Delta ^{+,2}(\rho)}  \sim  \rho ^{-1} \\
 \hline
\frac{\Delta_{1,3}^{+,2}(\rho)}{\Delta ^{+,2}(\rho)}  \sim
\rho^{-2} &\frac{\Delta_{2,3}^{+,2}(\rho)}{\Delta ^{+,2}(\rho)}
\sim
 \rho^{-2}e^{-\frac{\sqrt{3}}{2}\rho}  & \frac{\Delta_{3,3}^{+,2}(\rho)}{\Delta ^{+,2}(\rho)}  \sim  \rho^{-2}
   \\ \hline
\end{tabular}
\end{center}

\begin{center}
\renewcommand{\arraystretch}{2}
\begin{tabular}{||>{$}c<{$}| >{$}c<{$}|>{$}c<{$}||}\hline
\frac{\Delta_{1,1}^{+,3}(\rho)}{\Delta ^{+,3}(\rho)}  \sim \rho
^{-2}e^{-\frac{\sqrt{3}}{2}\rho} &
\frac{\Delta_{2,1}^{+,3}(\rho)}{\Delta ^{+,3}(\rho)}
 \sim  \rho ^{-2}e^{- \sqrt{3} \rho} & \frac{\Delta_{3,1}^{+,3}(\rho)}{\Delta ^{+,3}(\rho)}  \sim \rho ^{-2} \\
\hline \frac{\Delta_{1,2}^{+,3}(\rho)}{\Delta ^{+,3}(\rho)} \sim
e^{-\frac{\sqrt{3}}{2}\rho} &
\frac{\Delta_{2,2}^{+,3}(\rho)}{\Delta ^{+,3}(\rho)} \sim
 e^{-\frac{\sqrt{3}}{2}\rho} & \frac{\Delta_{3,2}^{+,3}(\rho)}{\Delta ^{+,3}(\rho)}  \sim  1 \\
 \hline
\frac{\Delta_{1,3}^{+,3}(\rho)}{\Delta ^{+,3}(\rho)}  \sim
\rho^{-1} &\frac{\Delta_{2,3}^{+,3}(\rho)}{\Delta ^{+,3}(\rho)}
\sim
 \rho^{-1}e^{-\frac{\sqrt{3}}{2}\rho}  & \frac{\Delta_{3,3}^{+,3}(\rho)}{\Delta ^{+,3}(\rho)}  \sim  \rho^{-1}
   \\ \hline
\end{tabular}
\end{center}

\begin{center}
\renewcommand{\arraystretch}{2}
\begin{tabular}{||>{$}c<{$}| >{$}c<{$}|>{$}c<{$}||}\hline
\frac{\Delta_{1,1}^{+,4}(\rho)}{\Delta ^{+,4}(\rho)}  \sim
\rho^{-2}e^{-\frac{\sqrt{3}}{2}\rho} &
\frac{\Delta_{2,1}^{+,4}(\rho)}{\Delta ^{+,4}(\rho)}
 \sim  \rho ^{-2}e^{- \sqrt{3} \rho} & \frac{\Delta_{3,1}^{+,4}(\rho)}{\Delta ^{+,4}(\rho)}  \sim \rho ^{-2} \\
\hline \frac{\Delta_{1,2}^{+,4}(\rho)}{\Delta ^{+,4}(\rho)}
\sim\rho ^{-1} & \frac{\Delta_{2,2}^{+,4}(\rho)}{\Delta
^{+,4}(\rho)} \sim
 \rho ^{-1}e^{-\frac{\sqrt{3}}{2}\rho} & \frac{\Delta_{3,2}^{+,4}(\rho)}{\Delta ^{+,4}(\rho)}  \sim  \rho ^{-1} \\
 \hline
\frac{\Delta_{1,3}^{+,4}(\rho)}{\Delta ^{+,4}(\rho)}  \sim
\rho^{-2} &\frac{\Delta_{2,3}^{+,4}(\rho)}{\Delta ^{+,4}(\rho)}
\sim
 \rho^{-2}e^{-\frac{\sqrt{3}}{2}\rho}  & \frac{\Delta_{3,3}^{+,4}(\rho)}{\Delta ^{+,4}(\rho)}  \sim  \rho^{-2}
   \\ \hline
\end{tabular}
\end{center}
Then (\ref{2.6-1})-(\ref{2.6-4}) follow consequently. $\Box$

\smallskip
We consider next the linear IBVP with homogeneous boundary conditions
\begin{equation}
\label{y-4}
\begin{cases}
z_t +z_{xxx} +\delta _kz=f (x,t), \quad x\in (0,L), \ t>0,\\
z(x,0)= \phi (x),\\
\mathcal{B}_{k,0}z =0
\end{cases}
\end{equation}
for $k=1,2,3,4.$  By the standard semigroup theory, for any $\phi \in L^2 (0,L)$, $f\in L^1_{loc} (\mathbb{R}^+; L^2 (0,L)),$  the IBVP (\ref{y-4}) admits a unique solution
$z\in C(\mathbb{R}^+; L^2 (0,L))$ which can be written as
\[  z(x,t)=W_{0 ,k}(t) \phi  + \int ^t_0 W_{0,k} (t-\tau ) f(\cdot, \tau ) d\tau \]
where  $W_{0,k}(t)$ is the $C_0$-semigroup associated with the IBVP (\ref{y-4}) with $f\equiv 0$.  Recall the solution of  the Cauchy problem of the linear KdV equation,
\begin{equation}\label{y-5}
        \begin{cases}
            w_t+w_{xxx} +\delta _k w=0, & x \in \r, \  t\geq 0,\\
             w(x,0)=\psi(x),& x \in \r,
        \end{cases}
    \end{equation}
has the explicit representation
\begin{equation}\label{RSol1}
v(x,t)= [W_{\r,k} (t)]\psi(x)= c \int_{\r}
e^{i\xi^3 t-\delta _kt}e^{ix \xi } \hat{\psi} (\xi)d\xi.
\end{equation}
Here $\hat{\psi }$ denotes the Fourier transform of $\phi $. In terms of the $C_0$-group  $W_{\r,k}(t)$ and the boundary integral operator  ${\cal W}_{bdr} ^{(k)}$ , we can have a more explicit representation of solutions of the IBVP (\ref{y-4}).

Let $s\geq 0$ be given and $ B_s : H^s( 0, L) \to  H^s(\r) $ be  the standard extension operator from  $H^s(0, L) $ to $H^s(\r) $. For any $\phi\in H^s (0,L)$ and
$f\in L^1 _{loc} (\r ^+; H^s (0,L))$
let \[ \phi ^*= B_s \phi\] and  \[ f^* = B_s f.\]  \begin{lem} For  given $\phi \in L^2 (0,L)$ and $f\in L^1 _{loc}( \r^+; L^2 (0,L))$, let
\[  q_k(x,t)= W_{\r,k} (t)\phi ^* + \int ^t_0 W_{\r,k} (t-\tau ) f^* (\tau) d\tau \]
and
\[ \vec{h}_k := \mathcal{B}_{k,0} q, \quad k=1,2,3,4 .\] Then the solution of  the IBVP (\ref{y-4}) can be written as
\[  z (x,t)=  W_{\r,k} (t)\phi ^* + \int ^t_0 W_{\r,k} (t-\tau ) f^* (\tau) d\tau - {\cal W}^{(k)}_{bdr}  \vec{h}_k .\]
\end{lem}

\subsection{Linear estimates}\label{sec22}
In this subsection we consider the following IBVP of the linear equations:
\begin{equation}\label{LSP}
        \begin{cases}
            v_t+ v_{xxx} +\delta _k v=f, \quad v(x,0)=\phi (x), & x \in (0,L), \  t\geq 0,\\
            \mathcal{B}_{k,0}v=\vec{h}(t),& t\geq 0
        \end{cases}
    \end{equation}
    and present various linear estimates for its solutions.
For given $s\geq 0$ and $T>0$, let us consider:
%\[ X^1_{s,T}:= H_0^s (0,L)\times \overbrace{H_0^{\frac{s+1}{3}}(0,T)\times
%H_0^{\frac{s+1}{3}} (0,T)\times H_0^{\frac{s}{3}}(0,T)}^{\mathcal{H}^1_{s,T}},\]
%\[ X^2_{s,T}:= H_0^s (0,L)\times \overbrace{H_0^{\frac{s+1}{3}}(0,T)\times
%H_0^{\frac{s}{3}} (0,T)\times H_0^{\frac{s-1}{3}}(0,T)}^{\mathcal{H}^2_{s,T}},\]
%\[ X^3_{s,T}:= H_0^s (0,L)\times  \overbrace{H_0^{\frac{s-1}{3}}(0,T)\times
%H_0^{\frac{s+1}{3}} (0,T)\times H_0^{\frac{s}{3}}(0,T)}^{\mathcal{H}^3_{s,T}},\]
%\[ X^4_{s,T}:= H_0^s (0,L)\times \overbrace{H_0^{\frac{s-1}{3}}(0,T)\times
%H_0^{\frac{s}{3}} (0,T)\times H_0^{\frac{s-1}{3}}(0,T)}^{\mathcal{H}^4_{s,T}}\]
%and
\[Z_{s,T}:=C([0,T];H^s (0,L))\cap L^2 (0,T; H^{s+1}(0,L))\]
and
$$X^k_{s,T}:=H_0^s (0,L)\times\mathcal{H}^s_{k} (0,T), \text{} \text{ for } k=1,2,3,4.$$
Recall that  when $f=0$ and $\phi=0$, the solution $v$ of the IBVP  \eqref{LSP}  can be written in the form
\[
v(x,t)=[W_{bdr}^{(k)}\vec{h}](x,t):= \sum_{j,m=1}^3
[W_{j,m}^{(k)}h_{m}](x,t),
\]
where
\[
 [W_{j,m}^{(k)}h](x,t) \equiv [U_{j,m}^{(k) }h](x,t) +\overline{[U_{j,m}^{(k)}h](x,t)}
\]
with
\[ [U_{j,m}^{(k)} h](x,t)\equiv \frac{1}{2\pi } \int ^{+\infty }_{0} e^{i\rho ^3
t} e^{\lambda ^+_j (\rho ) x}  \hat{h}^*_{j,m,k} (\rho) d\rho
\]
for $k=1,2,3,4$, $j=1,3, \ m=1,2,3 $ and
\[
 [U_{2,m}^{(k)} h](x,t) \equiv  \frac{1}{2\pi } \int ^{+\infty }_{0} e^{i \rho ^3
 t} e^{-\lambda ^+_2 (\rho ) (1-x)}    \hat{h}^*_{2,m,k} (\rho)d\rho \label{ge_3}
\]
for $k=1,2,3,4$ and $m=1,2,3$. Here \[\label{ge_4}\hat{h}^*_{j,m,k} (\rho)= 3\rho ^2 \frac{\Delta ^{+,k}_{j,m} (\rho ) } {\Delta
^{+,k} (\rho )} \hat{h}   ^+ (\rho ), \qquad \hat{h}^*_{2,m,k} (\rho)=3\rho ^2 \frac{\Delta ^{+,k}_{2,m} (\rho ) } {\Delta ^{+,k} (\rho
)} e^{\lambda ^+_2 (\rho )} \hat{h} ^+ (\rho )\]
for
$k=1,2,3,4$, $j=1,3 $ and $m=1,2,3$.
%Here $hat{h}^+ (\rho) = \hat{h} (i\rho ^3)$, $\Delta    ^{+,k}(\rho)$ and $\Delta ^{+,k}_{j,m} (\rho)$ are obtained from
  %  $\Delta ^k (s)$ and $\Delta _{j,m}^k (s)$ by replacing $s$ with $i \rho ^3  $ and  $\lambda _j ^+ (\rho)=\lambda _j (i \rho
  %  ^3 )$, for $k=1,2,3,4$.

%Here, for simplicity, we will consider the case when $k=4$, namely,
  %  \begin{equation}\label{w1b}
      %  \begin{cases}
         %   w_t+ w_{xxx}=f, \quad w(x,0)=\phi,  & x \in (0,L), \  t\geq 0,\\
           % w_{xx}(0,t)=h_{1,4}(t), \ w_x(L,t)=h_{2,4}(t), \ w_{xx}(L,t)=h_{3,4}(t), & t\geq 0,
       % \end{cases}
  %  \end{equation}
  %  the other cases follow similarly.
\begin{prop}\label{l1}
Let  $0\leq s\leq 3$ with
$s\neq\frac{2j-1}{2}, \text{} \text{}j=1,2,3,  $ and $T>0$ be given.   There exists a constant $C>0$  such that for any $\vec{h}\in\mathcal{H}^s_k (0,T)$ ,
\[ z_k= {\cal W}_{bdr}^(k) \vec{h}\]
satisfies
$$\|z_k\|_{\mathcal{Z}_{s,T}}+\sum_{j=0}^2\|\partial^j_xz_k\|_{L^{\infty}(0,L;H^{\frac{s+1-l}{3}}(0,T))}\leq C\|\vec{h}\|_{\mathcal{H}^s_{k} (0,T)}$$
for $k=1,2,3,4$ and $l=0,1,2$.
\end{prop}
\begin{proof}
We only consider the case that $\vec{h}=(h_1,0,0)$ and $k=4$; the proofs for the others cases are similar. Note that, the solution $z_4 $ can be written as
$$z_4(x,t)=w_1(x,t)+w_2(x,t)+w_3(x,t)$$
with
\[  w_j (x,t) := [W_{j,1}^{(4)}h_{1}](x,t) = [U_{j,1}^{(4) }h_1](x,t) +\overline{[U_{j,1}^{(4)}h_1](x,t)}, \quad j=1,2,3.\]
Let us prove Proposition  \ref{l1} for $w_1$.  It suffices to only consider
\[ w_1^+(x,t):= [U_{1,1}^{(4) }h_1](x,t)  =\frac{1}{2\pi } \int ^{+\infty }_{0} e^{i\rho ^3
t} e^{\lambda ^+_1 (\rho ) x}
\hat{h}^*_{1,1,4}(\rho )d\rho.\]
%Note that $\lambda _1^+ (\rho)= i\rho$ and as $\rho \to \infty$,
%\[  \frac{\Delta_{1,1}^{+,4}(\rho)}{\Delta ^{+,4}(\rho)}  \sim
%\rho^{-2}e^{-\frac{\sqrt{3}}{2}\rho}.\]

%Some straightforward calculations, for $\rho\to \infty$, show that the asymptotic behavior of the ratios $\frac{\Delta^{+,4}_{j,m}(\rho)}{\Delta ^{+,4}(\rho)}$, for
%$m,j=1,2,3$ as $\rho\to+\infty$ are:

%\begin{center}
%\renewcommand{\arraystretch}{2}
%\begin{tabular}{||>{$}c<{$}| >{$}c<{$}|>{$}c<{$}||}\hline
%\frac{\Delta_{1,1}^{+,4}(\rho)}{\Delta ^{+,4}(\rho)}  \sim
%\rho^{-2}e^{-\frac{\sqrt{3}}{2}\rho} &
%\frac{\Delta_{2,1}^{+,4}(\rho)}{\Delta ^{+,4}(\rho)}
 %\sim  \rho ^{-2}e^{- \sqrt{3} \rho} & \frac{\Delta_{3,1}^{+,4}(\rho)}{\Delta ^{+,4}(\rho)}  \sim \rho ^{-2} \\
%\hline \frac{\Delta_{1,2}^{+,4}(\rho)}{\Delta ^{+,4}(\rho)}
%\sim\rho ^{-1} & \frac{\Delta_{2,2}^{+,4}(\rho)}{\Delta
%^{+,4}(\rho)} \sim
 %\rho ^{-1}e^{-\frac{\sqrt{3}}{2}\rho} & \frac{\Delta_{3,2}^{+,4}(\rho)}{\Delta ^{+,4}(\rho)}  \sim  \rho ^{-1} \\
% \hline
%\frac{\Delta_{1,3}^{+,4}(\rho)}{\Delta ^{+,4}(\rho)}  \sim
%\rho^{-2} &\frac{\Delta_{2,3}^{+,4}(\rho)}{\Delta ^{+,4}(\rho)}
%\sim
% \rho^{-2}e^{-\frac{\sqrt{3}}{2}\rho}  & \frac{\Delta_{3,3}^{+,4}(\rho)}{\Delta ^{+,4}(\rho)}  \sim  \rho^{-2}
 %  \\ \hline
%\end{tabular}
%\end{center}

Applying \cite[Lemma 2.5]{BSZ03}, we have,
\begin{align*}
\sup_{t\in[0,T]}\|w^+_1(\cdot,t)\|^2_{L^2(0,L)} & \leq C\int_0^{\infty}\left|\hat{h}^*_{1,1,4}(\rho )\right|^2d\rho\\
&  \leq C\|h_{1}\|^2_{H^{-\frac{1}{3}}(\mathbb{R}^+)}
\end{align*}
and
\begin{align*}
\sup_{t\in[0,T]}\|\partial^3_xw^+_1(\cdot,t)\|^2_{L^2(0,L)}& \leq C\int_0^{\infty}\left|\lambda^+_1(\rho)\right|^6\left|\hat{h}^*_{1,1,4}(\rho )\right|^2d\rho\\
%& \leq C\int_0^{\infty}\mu^{4/3}\left|e^{-i\mu\tau} h^+_{1}(\tau)d\tau\right|^2d\mu\\
&  \leq C\|h_{1}\|^2_{H^{\frac{2}{3}}(\mathbb{R}^+)}.
\end{align*}
 By interpolation, we have
$$\sup_{t\in[0,T]}\|w^+_1(\cdot,t)\|^2_{H^s(0,L)}\leq C\|h_{1}\|^2_{H^{\frac{s-1}{3}}(\mathbb{R}^+)} $$
 for $0\leq s\leq3$. Furthermore, for $l=0,1,2$, let $\mu=\rho^3$, $\rho\geq 0$, then
\begin{align*}
\partial^l_xw_1(x,t)&=\frac{1}{2\pi}\int_0^{+\infty}\left(\lambda^+_1(\rho)\right)^le^{i\rho^3t}e^{\lambda_1^+(\rho)x} \hat{h}^*_{1,1,4}(\rho )d\rho\\
& = \frac{1}{2\pi} \int_0^{+\infty}\left(\lambda^+_1(\mu ^{\frac13}) \right) ^le^{i\mu t}e^{\lambda_1^+(\mu ^{\frac13})x}\hat{h}^*_{1,1,4}(\mu ^{\frac13}) \mu ^{-\frac23}d\mu.
\end{align*}
Applying Plancherel theorem, in time $t$, yields that, for all $x\in(0,L)$,
\begin{align*}
\|\partial^{l}_xw_1(x,\cdot)\|^2_{H^{\frac{s+1-l}{3}}(0,T)}&\leq C\int_0^{+\infty}\mu ^{\frac{2(s+1-l)}{3}}\left|\left(\lambda^+_1(\mu ^{\frac13}) \right) ^le^{\lambda_1^+(\mu ^{\frac13})x}\hat{h}^*_{1,1,4}(\mu ^{\frac13}) \mu ^{-\frac23}\right|^2d\mu\\
&\leq C\int_0^{+\infty}\left|(\lambda^+_1(\rho))^{l})\hat{h}^*_{1,1,4}(\rho )\right|^2\rho^{2s-2l}d\rho\\
& \leq C\int_0^{+\infty}\rho^{2s}\left|\hat{h}^*_{1,1,4}(\rho )\right|^2d\rho\\
& \leq C\|h_{1}\|^2_{H^{\frac{s-1}{3}}(\mathbb{R}^+)},
\end{align*}
for $l=0,1,2$. Consequently, for $0\leq s\leq3$ and $l=0,1,2$, we have
$$ \sup_{x\in(0,L)}\|\partial^l_xw_1(x,\cdot)\|_{H^{\frac{s+1-l}{3}}(0,T)}\leq C\|h_{1}\|_{H^{\frac{s-1}{3}}(\mathbb{R}^+)},$$
which ends the proof of Proposition \ref{l1} for $w_1$. The proof for $w_j$, $j=2,3$, are similar therefore will be omitted.
\end{proof}

Next we consider the following initial boundary-value problem:
    \begin{equation}\label{w1}
        \begin{cases}
            v_t+ v_{xxx}+ \delta _k v=f, \quad v(x,0)=\phi(x),  & x \in (0,L), \  t\geq 0,\\
            B_{k,0}v=0,& t\geq 0,
        \end{cases}
    \end{equation}
    for $k=1,2,3,4$.
Recall  that  for any  $s\in \r$,  $\psi \in H^s (\r)$ and $g\in L^1_{loc} (\r^+; H^s (\r))$, the Cauchy problem of  the following linear KdV equation posed on $\r$,
\begin{equation}\label{RP1}
        \begin{cases}
            w_t+w_{xxx}+ \delta _k w=g, & x \in \r, \  t\geq 0,\\
             w(x,0)=\psi(x),& x \in \r
        \end{cases}
    \end{equation}
 admits a unique  solution $v\in C(\r^+; H^s (\r))$ and  possess the  well-known  sharp Kato smoothing properties.
 \begin{lem}\label{FSP1b}
Let $T>0$, $L>0$ and $s\in \r$ be given. For any $\psi \in H^s(\mathbb{R})$, $g\in L^1(0,T;H^s(\mathbb{R}))$, the solution $w$ of the system \eqref{RP1} admits a unique solution
$w\in Z_{s,T}$
with
$$\partial^l_xw\in L^{\infty}_x(\mathbb{R};H^{\frac{s+1-l}{3}}(0,T)).$$
Moreover, there exists a constant $C>0$, depending only $s$  and $T$, such that
$$
\|w\|_{\mathbb{Z}}+\sum_{l=0}^2\|\partial^l_xw(x,\cdot)\|_{L_x^{\infty}(\mathbb{R};H^{\frac{s+1-l}{3}}(0,T))}\leq C\left(\|\psi \|_{H^s(\mathbb{R}^)}+\|g\|_{L^1(0,T;H^s(\mathbb{R}))} \right),$$
for $l=0,1,2$.
\end{lem}
\begin{cor} Let $0\leq s \leq 3$ and $T>0$ be given. For any $\phi \in H^s _0(0,L)$ and $g\in L^1 (0,T; H^s (\r)),$  let $\psi \in H^s (\r)$ be zero extension of $\phi $ from $(0,L)$ to $\r$. If
\[ \vec{q}_k: = \mathcal{B}_{k,0} w, \quad  k=1,2,3,4,\]
then
\[ \vec{q}_k \in {\cal H}^s_k (0,T) \]
and
\[ \| \vec{q}_k\|_{{\cal H}^s_k (0,T)} \leq C\left(\|\phi\|_{H^s(0,L)}+\|g\|_{L^1(0,T;H^s(\mathbb{R}))} \right)\]
for $k=1,2,3, 4$.
\end{cor}
The following two propositions follow from Proposition \ref{l1} and Lemma \ref{FSP1b}.
 \begin{prop}\label{p3} Let $T>0$ and $0\leq s\leq 3$  with
$s\neq\frac{2j-1}{2}, \text{} \text{}j=1,2,3, $ be given. There exists a constant $C>0$ such that for any  $(\phi,
    \vec{h})\in X^k_{s,T}$ and $f\in L^1(0,T; H^s (0,L))$, the IBVP \eqref{LSP} admits a unique
    solution $v\in Z_{s,T}$ satisfying
    \[ \| v\|_{Z_{s,T}}\leq
    C\left (\| (\phi , \vec{h})\|_{X^k_{s,T}} +\|f\|_{L^1(0,T; H^s (0,L))}\right ).\]
    \end{prop}
\begin{proof}
From lemmas \ref{l4-d}, \ref{l5-abcd}, \ref{FSP1b} and Proposition \ref{l1} the result holds.
\end{proof}

In addition, the solution $v$ of \eqref{LSP} posses the following  sharp Kato smoothing properties.
\begin{prop}\label{p4} Let $T>0$ and $0\leq s\leq 3$  with
$s\neq\frac{2j-1}{2}, \text{} \text{}j=1,2,3, $ be given. For any $\phi\in H^s(0,L)$, $f\in L^1(0,T;H^s(0,L))$ and $\vec{h}\in\mathcal{H}^k_{s,T}$, the solution $v$ of the system \eqref{LSP} satisfies
\begin{equation}
\sup_{x\in(0,L)}\|\partial^r_xv(x,\cdot)\|_{H^{\frac{s+1-r}{3}}(0,T)}\leq C\left (\| (\phi , \vec{h})\|_{X^k_{s,T}} +\|f\|_{L^1(0,T; H^s (0,L))}\right ),
\label{w7}
\end{equation}
for $r=0,1,2$.
\end{prop}

\section{Nonlinear problems}\label{sec3}
\setcounter{equation}{0}
In this section, we will  consider the IBVP of the nonlinear KdV equation on $(0,L)$ with the general boundary  conditions      \begin{equation}\label{3.1}
        \begin{cases}
           u_t+u_{xxx}+u_x +uu_x=0, & x\in (0,L),  \   t>0 \\
           u(x,0)=\phi(x),& x\in (0,L),\\
           \mathcal{B}_{k}\vec{v}(t)=\vec{h}(t), &t\geq0,
        \end{cases}
    \end{equation}
where  the boundary operators $\mathcal{B}_{k} $,  $k=1,2,3,4$,  are introduced in the introduction.

For given $s\geq 0$ and $T>0$, let
\[Y_{s,T}:=\left \{ w\in Z_{s,T};  \sum_{l=0}^2\|\partial^l_xw(x,\cdot)\|_{L_x^{\infty}(\mathbb{R};H^{\frac{s+1-l}{3}}(0,T))} < +\infty  \right \} \]
and
\[ \| w\|_{Y_{s,T}} := \left ( \|w\|_{Z_{s,T}}^2 +  \sum_{l=0}^2\|\partial^l_xw(x,\cdot)\|_{L_x^{\infty}(\mathbb{R};H^{\frac{s+1-l}{3}}(0,T))}^2 \right )^{\frac12}.\]
%, moreover, the definition of $s-$compatible can be seen in the Theorems \ref{th1}--\ref{th4}.
The next lemma is helpful in establishing the well-posedness of \eqref{3.1} whose proof can be found in \cite{BSZ03FiniteDomain, KrZh}.
\begin{lem}\label{lem1}
There exists a $C>0$  and $\mu >0$  sucht for any  $T>0$  and $u, \ v\in Y_{0,T}$,

\begin{equation}\label{Ys,t estimate}
\int\limits_0^T\|uv_x\|_{L^2(0,L)}\;d\tau\leq C
(T^{\frac{1}{2}}+T^{\frac13})\|u\|_{Y_{0,T}}\|v\|_{Y_{0,T}}
\end{equation}
%Moreover, if $0\leq s\leq 3$ there exists a constant $C>0$ such that for any  $T>0$ and
%$u, \ v\in {\mathcal Z}_{s,T}$,
%\begin{equation}\label{forcing}
%\| uv_x\|_{W^{\frac{s}{3},1} (0,T; L^2 (0,1)} \leq C
%(T^{\frac{1}{2}}+T^{\frac13})\|u\|_{{\mathcal Z}_{s,T}}\|v\|_{{\mathcal
%Z}_{s,T}}.
%\end{equation}
and
\[ \|{\cal B}_{k,1} v\|_{{\cal H}^0_k (0,T)} \leq CT^{\mu} \| {\cal B}_{k,1} v\|_{Y_{0,T}},\]
for $k=1,2,3,4.$
\end{lem}

Consider the following linear IBVPs
      \begin{equation}\label{3.1-a}
        \begin{cases}
           v_t+v_{xxx} +\delta _k v=f, & x\in (0,L),  \  t>0 \\
           v(x,0)=\phi(x),& x\in (0,L),\\
           \mathcal{B}_{k,0}v=\vec{h}, &t\geq0,
        \end{cases}
    \end{equation}
for $k=1,2,3,4$. The following lemma follows from the discussion in the Section \ref{sec2}, therefore, the proof will be omitted.

\begin{lem}\label{lem1-a}
Let $T>0$  be given.  There exists a constant $C>0$ such that for any $(\phi,\vec{h})\in X^k_{0,T}$ and $f\in L^1 (0,T; L^2 (0,L))$, the IBVP \eqref{3.1-a} admits a unique solution $v\in Y_{0,T}$ satisfying
    \[ \| v\|_{Y_{0,T}}\leq
    C\left (\| (\phi,\vec{h})\|_{X^k_{0,T}}+\|f\|_{L^1 (0,T; L^2 (0,L))}\right ),\]
for $k=1,2,3,4$.
\end{lem}

Next, we consider the following linearized IBVP associated to \eqref{3.1}
       \begin{equation}\label{3.2}
        \begin{cases}
           v_t+v_x+v_{xxx}+(a(x,t)v)_x =f, & x\in (0,L),  \   t>0 \\
           v(x,0)=\phi(x),& x\in (0,L),\\
           \mathcal{B}_{k}\vec{v}(t)=\vec{h}(t), & t\geq0,
        \end{cases}
    \end{equation}
for $k=1,2,3,4$ and $a(x,t)$ is a given function.
\begin{prop}\label{VC}
Let $T>0$   be given. Assume that $a\in Y_{0,T}$. Then for any $(\phi , \vec{h}) \in X^k_{0,T}$ and $f\in L^1(0,T;L^2(0,L))$, the IBVP \eqref{3.2} admits unique solution
\[ v\in Y_{0,T}.\]
Moreover, there exists a constant $C>0$ depending only on $T$ and $\|a\|_{Y_{0,T}}$ such that
 \[ \| v\|_{Y_{0,T}}\leq
    C\left (\| (\phi , \vec{h})\|_{X^k_{0,T}} +\|f\|_{L^1(0,T; L^2(0,L))}\right ).\]
\end{prop}
\begin{proof}
Let $r>0$ and $0<\theta\leq T$ be a constant to be determined. Set
$$S_{\theta,r}:=\{ u\in Y_{0,\theta}: \|u\|_{Y_{0,\theta}}\leq r\},$$
which is a bounded closed convex subset of $Y_{0,\theta}$.  For given $(\phi , \vec{h}) \in X^k_{0,T}$, $a\in Y_{0,T}$ and $f\in L^1 (0,T;L^2(0,L))$, define a map $\Gamma$ on $S_{\theta,r}$ by
 $$v=\Gamma(u)$$
for any $u\in S_{\theta,r}$   where  $v$ is  the unique solution of
      \begin{equation}\label{3.2c}
        \begin{cases}
        v_t+v_{xxx} +\delta _k v=-u_x-(a(x,t)u)_x +\delta _k u  & x\in (0,L) ,  \   t\geq0 \\
        v(x,0)=\phi(x),& x\in (0,L)\\
        \mathcal{B}_{k,0}\vec{v}(t)=\vec{h}(t)-{\cal B}_{k,1} u,& t\geq0,
        \end{cases}
    \end{equation}
    By  Lemma \ref{lem1-a} (see also Propositions \ref{p3} and \ref{p4}),  for any $u,w\in S_{\theta,r}$,
    \begin{eqnarray*}
     \| \Gamma (u)\|_{Y_{0,\theta}}
      &\leq &
      C_1\left (\|(\phi, \vec{h}) \|_{X^0_{k,T}} +\| f\|_{L^1 (0,T; L^2 (0,L))} \right ) \\ &&+C_2 \|B_{k,1} v\|_{{\cal H}^0_{k} (0,\theta)} + C_3  \| (av)_x\|_{L^1 (0,\theta; L^2 (0,L))} \\
     &\leq &
       C_1 \left (\|(\phi, \vec{h})\|_{X^0_{k,T}} +\| f\|_{L^1 (0,T; L^2 (0,L))}\right ) + \left (C_2 \theta ^{\mu} + \left [\theta ^{\frac12} +\theta ^{\frac13}\right ]
       \| a\|_{Y_{0,T} } \right ) \| v\| _{Y_{0, \theta }}   \end{eqnarray*}
 and
 $$\|\Gamma(w)-\Gamma(u)\|_{Y_{0,\theta}} \leq  \left (C_2 \theta ^{\mu} + \left [\theta ^{\frac12} +\theta ^{\frac13}\right ]
       \| a\|_{Y_{0,T} } \right )\|w-u\|_{Y_{0,\theta}}.$$
  Thus $\Gamma$ is a contraction mapping from $S_{r,\theta}$ to $S_{r,\theta}$  if one chooses $r$ and $\theta $ by
\[r=2C_0\left (\| (\phi , \vec{h})\|_{X^k_{0,T}} +\|f\|_{L^1 (0,T; L^2(0,L))} \right )\]
and
\[\left (C_2 \theta ^{\mu} + \left [\theta ^{\frac12} +\theta ^{\frac13}\right ]
       \| a\|_{Y_{0,T} } \right )\leq\frac{1}{2}.\]
Its fixed point $v=\Gamma(u)$ is desired solution of \eqref{3.2c} in the time interval $[0,\theta]$. Note that $\theta $ only depends on
$\| a\|_{Y_{0,T} }$
thus by standard extension argument, the solution $v$ can be extended to the time interval $[0,T]$. Thus, the proof is completed.
\end{proof}

    \smallskip
 Now, we turn to consider the well-posedness problem of the nonlinear IBVP \eqref{3.1}.
    % We first show that the IBVP \eqref{3.1} is locally well-posed in the space $H^s(0,L)$ for $s\in[0,3]$.

    \begin{thm} Let $s\geq 0$ with
$s\neq\frac{2j-1}{2}, \text{} \text{}j=1,2,3..., $  $T>0$  and $r>0$ be given.  There exists a $T^* \in (0,T]$  such that for any $(\phi, \vec{h})\in X^0_k (0,T)$  with
\[ \| (\phi, \vec{h})\|_{X^s_k (0,T)} \leq r,\]
the IBVP (\ref{3.1}) admits a unique solution $u\in Y_{s, T^*}$ . Moreover, the corresponding solution map is real analytic.
\end{thm}

\begin{proof}    We only prove the theorem in the case  of  $0\leq s\leq 3$.   When $s>3$  it follows from  a  standard procedure  developed in  \cite{BSZ02}.
 First we consider the case of $s=0$. As in the proof of Proposition \ref{VC}, let $r>0$ and $0<\theta\leq T$ be a constant to be determined. Set
\[ S_{\theta,r}:=\{u\in Y_{s,\theta}: \|u\|_{Y_{s,\theta}}\leq r\},\]
For given $(\phi, \vec{h})\in X^0_{k,T}$, define a map $\Gamma$ on $S_{\theta,r}$ by
$$v=\Gamma(u) \quad for \ u\in Y_{0,\theta}
$$ where $v$  is the  unique solution of
      \begin{equation}\label{3.1c}
        \begin{cases}
           v_t+v_{xxx} +\delta _k v=-u_x-uu_x +\delta _k u, & x\in (0,L),  \   t\geq0 \\
           v(x,0)=\phi(x),& x\in (0,L),\\
           \mathcal{B}_{k} \vec{v}(t)=\vec{h}(t),& t\geq0 .
        \end{cases}
    \end{equation}
  By Proposition \ref{VC},  for any $u, \ w\in S_{\theta ,r}$,
  $$\|\Gamma (u)\|_{Y_{0,\theta}}\leq C_0\|(\phi,\vec{h})\|_{X^k_{0,T}}+ C_1 \theta \| u\|_{Y_{0,\theta }}+ C_2 \left (\theta ^{1/3} +\theta ^{1/2} \right ) \|u\|_{Y_{0,\theta}}^2$$
  and
  \[ \| \Gamma (u)-\Gamma (w)\|_{Y_{0,\theta}} \leq C_1 \theta \| u-w\|_{Y_{0,\theta }} +\frac{ C_2}{2} \left (\theta ^{1/3} +\theta ^{1/2} \right ) \|u+w\|_{Y_{0,\theta}}\|u-w\|_{Y_{0,\theta}} .\]
  Choosing $r$ and $\theta $  with
  \[  r=2C_0\|(\phi,\vec{h})\|_{X^k_{0,T}}, \qquad  C_1 \theta + C_2 \left (\theta ^{1/3} +\theta ^{1/2} \right )r \leq \frac12, \]
  $\Gamma$ is a contraction whose critical point is the desired solution.

Next we consider the case of $s=3$.  Let $v=u_t$ we have $v$ solves
\begin{equation}\label{3.2-1}
        \begin{cases}
           v_t+v_x+v_{xxx}+(a(x,t)v)_x=0, & x\in (0,L),  \   t>0 \\
           v(x,0)=\phi ^ *(x),& x\in (0,L),\\
           \mathcal{B}_{k}\vec{v}(t)=\vec{h}'(t), & t\geq0,
        \end{cases}
    \end{equation}
    where $\phi^* (x) =-\phi' (x)-\phi'''(x)$ and $a(x,t)=\frac12 u (x,t)$. Applying Proposition \ref{VC} implies that $v=u_t \in Y_{0,T^*}$. Then it
    follows from the equation
    \[ u_t +u_x +uu_x +u_{xxx}=0 \]
    that $u_{xxx}\in Y_{0,T^*} $ and  $u\in Y_{3,T^*}$.  The case of $0< s< 3$ follows using Tartart's nonlinear interpolation theory \cite{tartar} and the proof is archived.
 \end{proof}

\section{Concluding remarks}\label{sec4}
\setcounter{equation}{0}
In this paper we have studied   the  nonhomogenous boundary value problem of the KdV equation on the finite interval (0,L) with general boundary conditions,
\begin{equation}\label{4.1}
\begin{cases}
 u_t+u_x+u_{xxx}+uu_x=0, \qquad  0<x<L, \
 t>0\\
 u(x,0)=\phi(x)\\
{\cal B}_ku=\vec{h}
 \end{cases}
 \end{equation}
 and  have shown that the IBVP (\ref{4.1}) is locally well-posed in the  space $H^s (0,L)$   for any $s\geq 0$  with
$s\neq\frac{2j-1}{2}, \text{} \text{}j=1,2,3..., $ and $(\phi, \vec{h})\in X^s_{k,T}$.
 Two important tools have played indispensable  roles in   approach; one is the explicit representation of the boundary integral operators ${\cal W}_{bdr} ^{(k)}$ associated to the IBVP (\ref{4.1}) and the other one  is the sharp Kato smoothing property. We have obtained our results by first investigating  the associated linear IBVP
 \begin{equation}\label{4.1-1}
\begin{cases}
 u_t+u_{xxx} +\delta _k u=f, \qquad  0<x<L, \
 t>0,\\
 u(x,0)=\phi(x),\\
{\cal B}_{k,0}u=\vec{h}. \end{cases}
 \end{equation}
 The local well-posedness of the nonlinear IBVP \eqref{4.1} follows via {\em contraction mapping principe}.
 %Thus the local well-posedness results  for the  nonlinear IBVP (\ref{4.1}) established in this paper can be viewed as a  regular small perturbation of the well-posedness  result of the linear IBVP (\ref{4.1-1}).

 While the results reported in this paper has significantly improved the  earlier  works on general boundary value problems of the KdV equation on a finite interval,
  there are still many questions to be addressed  regarding the IBVP (\ref{4.1}).  Here we list a few of them which are most interesting to us.

 \begin{itemize}
 \item[(1)]  {\em Is the IBVP (\ref{4.1}) globally well-posed in the space $H^s (0,L)$ for some $s\geq 0$ or equivalently,  does  any solution of the IBVP (\ref{4.1})  blow up in the some space $H^s (0,L)$ in finite time?}

 \smallskip
 It is not clear  if the IBVP (\ref{4.1}) is globally well-posed  or not even in the  case of $\vec{h}\equiv 0 $.   It follows from our results that a solution $u$ of the IBVP (\ref{4.1}) blows up in  the space $H^s (0,L)$   for some $s\geq 0$ at  a finite time $T>0$ if and only   if
 \[ \lim_{t\to T^-}\| u(\cdot, t) \|_{L^2 (0,L)} =+\infty .\]
 Consequently, it suffices to establish a global a priori  $L^2 (0,L)$ estimate
 \begin{equation}\label{priori} \sup _{0\leq t< \infty} \|u(\cdot, t)\|_{L^2 (0,L)} < +\infty ,\end{equation}
 for solutions of the IBVP (\ref{4.1}) in order to obtain the global well-posedness of the IBVP (\ref{4.1}) in the space $H^s (0,L)$ for any $s\geq 0$. However, estimate (\ref{priori}) is known to be held only in one case
 \[ \begin{cases}
 u_t+u_x +uu_x + u_{xxx}=f, \qquad  0<x<L, \
 t>0\\
 u(x,0)=\phi(x)\\
u(0,t)=h_1 (t), \ u(L,t)= h_2 (t), \ u_x (L,t) =h_3 (t). \end{cases}
\]

\item[(2)] {\em Is the IBVP well-posed in the space $H^s (0,L)$  for some $s\leq -1$?}

\smallskip
We have shown that the IBVP (\ref{4.1}) is locally well-posed in the space $H^s (0,L)$ for any $s\geq 0$. Our results can also be extended to the case of $-1< s\leq 0$ using the same approach developed in \cite{bsz-finite}.   For the pure initial value problems  (IVP) of the KdV equation posed on the whole line $\r$ or on torus $\mathbb{T}$,
\begin{equation}
\label{p-1}
u_t +uu_x +u_{xxx}=0, \quad u(x,0)= \phi (x), \quad x, \ t\in \r
\end{equation}
and
\begin{equation}
\label{p-2}
u_t +uu_x +u_{xxx}=0, \quad u(x,0)= \phi (x), \quad x \in \mathbb{T}, \ t\in \r ,
\end{equation}
it is well-known that the IVP (\ref{p-1}) is  well-posed in the space $H^s (\r)$ for any $s\geq -\frac34$ and is  (conditionally) ill-posed in the
space $H^s (\r) $ for any $s< -\frac34$ in the sense the corresponding solution map cannot be uniformly continuous.  As for the IVP (\ref{p-2}), it is well-posed in the space $H^s (\mathbb{T}) $ for any $s\geq -1$.   The solution map corresponding to the IVP (\ref{p-2}) is real analytic when $s>-\frac12$, but only continuous (not even  locally uniformly continuous) when $-1\leq s<-\frac12$. Whether  the IVP (\ref{p-1})  is well-posed  in the space $H^s(\r)$ for any $s<-\frac34$ or the IVP (\ref{p-2}) is well-posed in the space $H^s (\mathbb{T})$  for any $s< -1$ is still an open question. On the other hand, by contrast, the IVP of the KdV-Burgers equation
\[ u_t +uu_x +u_{xxx}-u_{xx}=0, \quad u(x,0)=\phi (x), \quad x\in \r, \ t>0 \]
is known to  be well-posed in the  space $H^s(\r) $ for any $s\geq -1$,  but  is known to be ill-posed for any $s<-1$. We conjecture that the IBVP (\ref{4.1}) is ill-posed in the space $H^s (0,L)$ for any $s<-1$.

\item[(3)] While the approach developed  in this paper can be used to study the nonhomogeneous boundary value problems of  the KdV equation on $(0,L)$ with quite general boundary conditions,  there are still  some boundary value  problems of the KdV equation that our approach do not work.  Among them the following two boundary value problems of the KdV equation on $(0,L)$ stand out:
\begin{equation}
\label{p-3}
\begin{cases}
u_t +uu_x +u_{xxx}=0, \quad x\in (0,L)\\
  u(x,0)= \phi (x), \\
  u(0,t)=u(L,t), \ u_x (0,t)=u_x (L,t), \ u_{xx} (0,t)= u_{xx}(L,t)
  \end{cases}
\end{equation}
and
\begin{equation}
\label{p-4}
\begin{cases}
u_t +uu_x +u_{xxx}=0, \quad x\in (0,L),\\
  u(x,0)= \phi (x),  \\
  u(0,t)=0, \ u (L,t)=0, \ u_{x} (0,t)= u_{x}(L,t) .
  \end{cases}
\end{equation}
A common feature for these two boundary value problems is that the $L^2-$norm of their solutions are conserved:
\[ \int ^L_0 u^2 (x,t) dx =\int ^L_0 \phi ^2 (x) dx \qquad \mbox{for any $t\in \r$}.\]
The IBVP (\ref{p-3}) is equivalent  to the IVP (\ref{p-2}) which was  shown by Kato \cite{Kato79,Kato83} to be well-posed in the space $H^s (\mathbb{T})$ when $s>\frac32$  as early as in the late 1970s.    Its well-posedness in the space $H^s (\mathbb{T})$ when $s\leq \frac32$ , however, was established 24  years later in the celebrated work of Bourgain \cite{Bourgain93a, Bourgain93b} in 1993. As for the IBVP (\ref{p-4}),  its associated linear problem
\begin{equation}
\label{p-5}
\begin{cases}
u_t  +u_{xxx}=0, \quad x\in (0,L),\\
  u(x,0)= \phi (x),
   u(0,t)=0, \\ u (L,t)=0, \ u_{x} (0,t)= u_{x}(L,t)
  \end{cases}
\end{equation}
has been shown by Cerpa  (see, for instance,  \cite{cerpatut})  to be well-posed  in the space $H^s (0,L)$ forward and backward in time. However, whether the nonlinear IBVP (\ref{p-4}) is well-posed in the space $H^s (0,L)$ for some $s$ is still unknown.
 \end{itemize}

\vglue 0.4 cm
\noindent\textbf{Acknowledgments:} Roberto Capistrano-Filho was supported by CNPq (Brazilian Technology Ministry), Project PDE, grant 229204/2013-9 and partially supported by CAPES (Brazilian Education Ministry) and Bing-Yu Zhang was partially supported by a grant from the Simons Foundation (201615), NSF of China (11571244, 11231007).
%, and PCSIRT (Chinese Education Ministry) under grant IRT 1273.
Part of this work was done during the post-doctoral visit of the first author at the University of Cincinnati, who thanks the host institution for the warm hospitality.

\end{document}